\documentclass[12pt,leqno,a4paper]{amsart}

\usepackage{amsmath}
\usepackage{amssymb}
\usepackage{enumerate}
\usepackage{txfonts}
\usepackage{mathrsfs}
\usepackage{hyperref}
\usepackage[all]{xy}
\usepackage{xcolor}
\usepackage{framed}

\newcommand{\CC}{\mathbb{C}}
\newcommand{\FF}{\mathbb{F}}
\newcommand{\ZZ}{\mathbb{Z}}

\newcommand{\cA}{\mathcal{A}}

\newcommand{\cD}{\mathcal{D}}
\newcommand{\cG}{\mathcal{G}}
\newcommand{\cH}{\mathcal{H}}
\newcommand{\cI}{\mathcal{I}}
\newcommand{\cO}{\mathcal{O}}
\newcommand{\cP}{\mathcal{P}}
\newcommand{\cQ}{\mathcal{Q}}
\newcommand{\cT}{\mathcal{T}}
\newcommand{\cU}{\mathcal{U}}
\newcommand{\cX}{\mathcal{X}}

\newcommand{\fS}{\mathfrak{S}}
\newcommand{\Cl}{\mathfrak{Cl}}

\newcommand{\ty}[1]{\mathsf{#1}}

\newcommand{\Alp}{\operatorname{Alp}}
\newcommand{\Aut}{\operatorname{Aut}}
\newcommand{\bl}{\operatorname{bl}}
\newcommand{\Char}{\operatorname{Char}}

\newcommand{\dz}{\operatorname{dz}}

\newcommand{\Ind}{\operatorname{Ind}}
\newcommand{\Irr}{\operatorname{Irr}}
\newcommand{\Lin}{\operatorname{Lin}}

\newcommand{\Res}{\operatorname{Res}}
\newcommand{\trace}{\operatorname{trace}}
\newcommand{\rdz}{\operatorname{rdz}}
\newcommand{\mrO}{\operatorname{O}}
\newcommand{\Z}{\operatorname{Z}}
\newcommand{\N}{\operatorname{N}}
\newcommand{\C}{\operatorname{C}}

\newcommand{\GL}{\operatorname{GL}}

\newcommand{\tvhi}{\widetilde{\vhi}}
\newcommand{\hvhi}{\widehat{\vhi}}

\newcommand{\tchi}{\widetilde{\chi}}

\newcommand{\hchi}{\widehat{\chi}}
\newcommand{\ttheta}{\widetilde{\theta}}

\newcommand{\hG}{\widehat{G}}
\newcommand{\tB}{{\widetilde{B}}}

\newcommand{\tD}{\widetilde{D}}
\newcommand{\tG}{\widetilde{G}}
\newcommand{\tH}{\widetilde{H}}
\newcommand{\tM}{\widetilde{M}}
\newcommand{\tN}{\widetilde{N}}

\newcommand{\tR}{\widetilde{R}}

\let\al=\alpha
\let\eps=\epsilon
\let\la=\lambda
\let\vhi=\varphi
\let\ze=\zeta
\let\ti=\times

\theoremstyle{theorem}

\newtheorem{thm}{Theorem}[section]
\newtheorem{lem}[thm]{Lemma}
\newtheorem{prop}[thm]{Proposition}
\newtheorem{cor}[thm]{Corollary}

\theoremstyle{definition}
\newtheorem{defn}[thm]{Definition}

\newtheorem{rmk}[thm]{Remark}

\numberwithin{equation}{section}

\begin{document}

\title[Character triples and weights]{Character triples and weights}

\author{Zhicheng Feng}
\address{Shenzhen International Center for Mathematics and Department of Mathematics, Southern University of Science and Technology, Shenzhen 518055, China}
\makeatletter
\email{fengzc@sustech.edu.cn}
\makeatother


\begin{abstract}
We define a new relation between character triples and prove some Clifford theory properties for weights in terms of character triples.
\end{abstract}



\date{\today}

\maketitle


\pagestyle{myheadings}

\section{Introduction}

Local-global conjectures establish profound connections between the representation theory of finite groups and their local subgroups.
Since the landmark reduction of the McKay conjecture by Isaacs, Malle, and Navarro \cite{IMN07} in 2007, significant progress has been made in reducing some local-global conjectures, proposed by Alperin, Brauer, and Dade etc, to questions on finite simple groups;
see for instance the works of Navarro--Tiep \cite{NT11}, Navarro--Sp\"ath \cite{NS14} and Sp\"ath \cite{Sp13b,Sp13,Sp17}. 
Combining these reduction theorems with the classification of finite simple groups has recently led to breakthroughs in longstanding problems, including the proofs of the Brauer height zero conjecture and the McKay conjecture (cf. \cite{CS24,KM13,MNST24,MS16,Ru24}).

Even though the inductive conditions of local-global conjectures are often much stronger than the original conjectures, the reduction framework provides powerful machinery in investigating the conjectures.
In \cite{Sp17b}, Sp\"ath reformulated some of the inductive conditions in terms of character triples, using the central and block isomorphisms between character triples which were first introduced by Navarro and Sp\"ath \cite{NS14}.

The central and block isomorphisms between character triples have been extensively studied, and have also proved helpful in the verification of the inductive conditions, for example, in the reduction of the inductive Alperin--McKay condition and the inductive blockwise Alperin weight condition to quasi-isolated blocks by Ruhstorfer \cite{Ru22}, and by the author, Li and Zhang \cite{FLZ22},
in the verification of the inductive Alperin--McKay condition for quasi-isolated 2-blocks by Ruhstorfer \cite{Ru24}, and in the proof of the inductive blockwise Alperin weight condition for type $\mathsf A$ by the author, Li and Zhang \cite{FLZ23}.
Moreover, Rossi \cite{Ro23,Ro24} and Mart\'inez--Rizo--Rossi \cite{MRR23} achieved the self-reducing forms for several local-global conjectures.
Among these works, Clifford properties in central and block isomorphisms between character triples and modular character triples are crucial.

In this paper, we focus more on the Alperin weight conjecture, which Alperin proposed \cite{Al87} in 1986 and reduced to simple groups by Navarro--Tiep \cite{NT11} and Sp\"ath \cite{Sp13}.
We generalize the central and block isomorphisms between character triples and prove Clifford theory properties for weights in terms of character triples.

Recall that some Clifford theory properties for weights in terms of modular character triples have already been achieved in \cite{FLZ22,FLZ23}.
In the verification of the inductive blockwise Alperin weight condition for groups of Lie type in non-defining characteristics, we often need unitriangular basic sets to describe the behavior of Brauer characters, since the ordinary characters of groups of Lie type were better understood by Deligne--Lusztig theory (cf. \cite{Lu84}). 
For this reason, we consider the ordinary character triples and the Clifford theory properties for weights.
To do this, we need a new relation between character triples, especially in the case that the quotient groups have orders divisible by $p$.

The results in this paper proved helpful in investigating Alperin weight conjecture, especially for groups of Lie type.
We mention that this paper can provide part of the crucial techniques in establishing the relation between weights and generic weights of groups of Lie type at good primes in the author's joint work with Malle and Zhang \cite{FMZ24}. 

The structure of this paper is as follows.
Section~\ref{sec:Preli} gives some preliminary results for the induced blocks and Dade's ramification groups.
In Section~\ref{sec:new-relation}, we generalize the central and block isomorphisms between character triples, as well as generalize some properties.
In Section~\ref{sec:construction}, we prove results on constructing the new relation for more character triples. 
Section~\ref{sec:conclusion} concludes with some Clifford theory properties for weights in terms of character triples.

\section{Induced blocks}   \label{sec:Preli}

We first give some notation used in this paper.

\subsection{Notation}

If a group $G$ acts on a set $X$, we let $G_x$ denote the stabiliser of
$x\in X$ in $G$, and analogously we denote the setwise stabiliser of
$X'\subseteq X$ in $G$ by $G_{X'}$.
If $H\le G$, then we denote by $X/\!\sim_H$ the set of $H$-orbits on $X$.
Moreover, if a group $G$ acts on two sets $X$, $Y$ and $x\in X$, $y\in Y$, we
denote by $G_{x,y}$ the stabiliser of $y$ in $G_x$. For a positive integer $n$,
we denote the symmetric group on $n$ symbols by $\fS_n$.

Suppose that $G$ is a finite group and $H\le G$.
We denote the restriction of $\chi\in\Irr(G)$ to $H$ by $\Res_H^G(\chi)$, and
$\Ind^G_H(\theta)$ denotes the character induced from $\theta\in\Irr(H)$ to $G$.
As usual, for $\theta\in\Irr(H)$ the set of irreducible constituents of
$\Ind_H^G(\theta)$ is denoted by $\Irr(G\mid\theta)$, while $\Irr(H\mid\chi)$
denotes the set of irreducible constituents of $\Res_H^G(\chi)$ for
$\chi\in\Irr(G)$. For a subset $\cH\subseteq\Irr(H)$, we define
\[\Irr(G\mid\cH)=\bigcup_{\theta\in\cH}\Irr(G\mid\theta)\]
and for a subset $\cG\subseteq\Irr(G)$, we define
\[\Irr(H\mid\cG)=\bigcup_{\chi\in\cG}\Irr(H\mid\chi).\]
Additionally, for $N \unlhd G$, we sometimes identify the characters of $G/N$
with the characters of $G$ whose kernel contains $N$.

Let $p$ be a prime number. Throughout this paper, all modular representations considered
are with respect to $p$. For $\chi\in\Irr(G)$, the $p$-block of $G$
containing $\chi$ is denoted by $\bl(\chi)$, which is also denoted by
$\bl_G(\chi)$ where we add a subscript to indicate the ambient group $G$.
If $b$ is a union of blocks of $G$, then we write
$\Irr(b)=\bigcup_{B\in b}\Irr(B)$. For a block $b$ of a subgroup $H\le G$ we
denote by $b^G$ the induced block of $G$, when it is defined.

Denote by $\dz(G)$ the set of irreducible characters of $G$ of ($p$-)defect
zero. Let $N\unlhd G$ and $\theta\in\Irr(N)$. We set
\[\rdz(G\mid\theta)
=\{\,\chi\in\Irr(G\mid\theta)\mid \chi(1)_p/\theta(1)_p=|G/N|_p \,\}.\]
If moreover $\theta\in\dz(N)$, then $\rdz(G\mid\theta)\subseteq\dz(G)$, and
then we also write $\dz(G\mid\theta)$ for $\rdz(G\mid\theta)$.

Let $\Lin(G)$ denote the set of linear characters of $G$, which can be seen
as a multiplicative group. Then $\Lin(G)$ acts on $\Irr(G)$ by multiplication.
The Hall $p'$-subgroup of $\Lin(G)$ is denoted $\Lin_{p'}(G)$.

Denote by $\mrO_p(G)$ the largest normal $p$-subgroup of $G$. Similarly,
$\mrO_{p'}(G)$ denotes the largest normal $p'$-subgroup of $G$.
If $A$ is an abelian group, then we also write $A_p$ for its Sylow
$p$-subgroup and write $A_{p'}$ for its Hall $p'$-subgroup;
note that $A_p=\mrO_p(A)$ and $A_{p'}=\mrO_{p'}(A)$.


\subsection{Induced blocks}\label{subsec:induced-blocks}

Let $\cO$ be the ring of algebraic integers in $\CC$ and $I$ be a maximal ideal
of $\cO$ with $p \cO \subseteq I$. Then the quotient $\FF=\cO/I$ is an
algebraically closed field of characteristic $p$, see \cite[\S 2]{Na98}.
Let ${}^*:\cO\to\FF$ be the natural ring homomorphism.

Let $G$ be a finite group.
For every subset $C\subseteq G$, we denote by $C^+$ the sum $\sum_{x\in C}x$ which is seen as an element of $\FF G$ or $\ZZ G$.
If $g\in G$, we denote by $\Cl_G(g)$ the conjugacy class of $G$ containing $g$.
For every $\chi\in\Irr(G)$ we denote by $\la_\chi\colon\Z(\FF G)\to\FF$ the associated central character. 
Recall that \[\la_\chi(\Cl_G(g)^+)=(|\Cl_G(g)|\chi(g)/\chi(1))^*\] for $g\in G$.
For a block $B$ of $G$, the central function $\la_B=\la_\chi$ for any $\chi\in\Irr(B)$.

\begin{lem}\label{lemma:blockinductions1}
	Let $N\unlhd G$, $H\le G$ and $M=H\cap N$.
	Suppose that there exist characters $\ttheta\in\Irr(G)$ and $\tvhi\in\Irr(H)$ with $\Res^G_{N}(\ttheta)\in\Irr(N)$ and $\Res^H_M(\tvhi)\in\Irr(M)$.
	Assume that $\bl(\tvhi)^G$ is defined and equals $\bl(\ttheta)$.
	For $N\le J\unlhd G$, if $\bl(\Res^H_{J\cap H}(\tvhi))^J$ is defined, then $\bl(\Res^H_{J\cap H}(\tvhi))^J=\bl(\Res^G_J(\ttheta))$.
\end{lem}

\begin{proof}
	Since $J\unlhd G$, if $\bl(\Res^H_{J\cap H}(\tvhi))^J$ is defined, then by \cite[Lemma~3.2]{KS15}, $\bl(\ttheta)=\bl(\tvhi)^G$ covers $\bl(\Res^H_{J\cap H}(\tvhi))^J$.
	On the other hand, since $\Res^G_J(\ttheta)$ is irreducible, $\bl(\ttheta)$ covers a unique block of $J$, that is, $\bl(\Res^G_J(\ttheta))$.
	Thus $\bl(\Res^H_{J\cap H}(\tvhi))^J=\bl(\Res^G_J(\ttheta))$.	
\end{proof}

\begin{lem}\label{lem:central-char}
	Let $N\unlhd G$ and $\ttheta\in\Irr(G)$ with $\Res^G_N(\ttheta)\in\Irr(N)$.
	Let $\overline G=G/N$, $x\in G$ and $\ttheta_x=\Res^G_{\langle N,x\rangle}(\ttheta)$.
	\begin{enumerate}[\rm(a)]
		\item $\la_{\ttheta_x}((\Cl_G(x)\cap xN)^+)=\la_{\ttheta_{y}}((\Cl_G(y)\cap yN)^+)$ for every $y\in\Cl_G(x)$.
		\item Let $\overline x=xN\in\overline G$, $\overline\eta\in\Irr(\overline G)$ and $\eta$ be the inflation of $\overline\eta$ to $G$.
		Then \[\la_{\ttheta\eta}(\Cl_G(x)^+)=\la_{\ttheta_x}((\Cl_G(x)\cap xN)^+)\la_{\overline\eta}(\Cl_{\overline G}(\overline x)^+).\]
	\end{enumerate}
\end{lem}	

\begin{proof}
	Let $L_x/N=\C_{\overline G}(\overline x)$. Then $\Cl_{L_x}(x)=\Cl_G(x)\cap xN$ is a union of $\langle N,x\rangle$-conjugacy classes.
	Let $y\in \Cl_G(x)$ and let $g\in G$ such that $y=x^g$. Denote by $c_g$ the conjugation induced by $g$.
	Then $c_g$ gives a bijection $L_x\to L_{y}$ and from this we get a bijection $c_g\colon\Cl_{L_x}(x)\to\Cl_{L_{y}}(y)$.
	So $|\Cl_G(x)\cap xN|=|\Cl_G(y)\cap yN|$ and this implies (a).
	Part (b) is \cite[Lemma~2.2]{Sp17}.
\end{proof}

The following result generalizes \cite[Prop.~2.3]{NS14}.

\begin{prop}   \label{prop:induced-extensions}
	Let $N\unlhd G$, $H\le G$ and $M=H\cap N$. Let $\ttheta\in\Irr(G)$ and
	$\tvhi\in\Irr(H)$ with $\theta:=\Res^G_{N}(\ttheta)\in\Irr(N)$ and
	$\vhi:=\Res^H_M(\tvhi)\in\Irr(M)$.
	Assume that for every $N\le J\le G$ with $J/N$ cyclic, the block
	$\bl(\Res^{H}_{J\cap H}(\tvhi))^J$ is defined and equals
	$\bl(\Res^G_J (\ttheta))$. Then for every $\eta\in\Irr(G)$ with
	$N\le\ker(\eta)$ and $\eta'=\Res^G_H(\eta)\in\Irr(H)$, the block
	$\bl(\tvhi\eta')^G$ is defined and equals $\bl(\ttheta\eta)$.
\end{prop}	

\begin{proof}
	It is sufficient to show 
	\begin{equation}\label{equ:central-char-clifford}
		\la_{\ttheta\eta}(\Cl_G(x)^+)=\la_{\tvhi\eta'}((\Cl_G(x)\cap H)^+) 
		\addtocounter{thm}{1}\tag{\thethm}
	\end{equation}
	for any $x\in G$. Denote by $\overline\eta$ (resp. $\overline\eta'$) the
	character of $\overline G=G/N$ (resp. $\overline H=H/M$) corresponding to
	$\eta$ (resp. $\eta'$) so that $\overline\eta'=\Res^{\overline G}_{\overline H}(\overline\eta)$.
	
	Fix $x\in G$. By Lemma~\ref{lem:central-char}(b), 
	\begin{equation}\label{equ:cen-char-ttheta-eta}
		\la_{\ttheta\eta}(\Cl_G(x)^+)=\la_{\ttheta_x}((\Cl_G(x)\cap xN)^+)\la_{\overline \eta}(\Cl_{\overline G}(\overline x)^+).
		\addtocounter{thm}{1}\tag{\thethm}
	\end{equation}
	Write $\Cl_{\overline G}(\overline x)\cap\overline H=\Cl_{\overline H}(\overline y_1')\sqcup\cdots\sqcup\Cl_{\overline H}(\overline y_s')$ for $y_i'\in H$.
	Then $y_i'N\cap H=y_i'M$ for $1\le i\le s$ and \[\Cl_G(x)\cap H=\coprod_{i=1}^s\coprod_{j=1}^{a_i}\Cl_H(x_{ij}')\]
	where $x_{ij}'\in H$ satisfy $\overline x_{ij}'\in \Cl_{\overline H}(\overline y_i')$.
	For $1\le j\le a_i$, the set $\{ gM\mid g\in\Cl_H(x_{ij}')\}$ consists of the conjugacy class in $H/M$ containing $y_i'M$ and thus we can choose $x_{ij}'$ satisfying $x_{ij}'M=y_i'M$.
	In addition, $\Cl_G(x)\cap y_i'M\subseteq\coprod\limits_{j=1}^{a_i}\Cl_H(x_{ij}')$.
	
	By assumption,
	\[\la_{\ttheta_{y_{i}'}}((\Cl_G(x)\cap y_i'N)^+)=\la_{\tvhi_{y_{i}'}}((\Cl_G(x)\cap y_i'M)^+),\] where $\tvhi_{y_{i}'}=\Res^H_{\langle M,y_{i}'\rangle}(\tvhi)$. 
	Thus according to Lemma~\ref{lem:central-char}(a), 
	\[\la_{\tvhi_{y_{i}'}}((\Cl_G(x)\cap y_i'M)^+)=\la_{\tvhi_{y_k'}}((\Cl_G(x)\cap y_k'M)^+)\] for every $1\le i,k\le s$.
	Now
	\[\la_{\tvhi\eta'}(\Cl_H(x_{ij}')^+)=\la_{\tvhi_{x_{ij}'}}((\Cl_H(x_{ij}')\cap y_i'M)^+)\la_{\overline\eta'}(\Cl_{\overline H}(\overline y_i')^+)\]
	and so
	\[\la_{\tvhi\eta'}\bigg(\bigg(\coprod\limits_{j=1}^{a_i}\Cl_H(x_{ij}')\bigg)^+\bigg)=\la_{\tvhi_{y_{i}'}}((\Cl_G(x)\cap y_i'M)^+)\la_{\overline\eta'}(\Cl_{\overline H}(\overline y_i')^+).\]
	Therefore,
	\begin{equation}\label{equ:cen-char-tvhi-eta}
		\la_{\tvhi\eta'}((\Cl(x)\cap H)^+)=\la_{\tvhi_{y_{1}'}}((\Cl_G(x)\cap y_1'M)^+)\la_{\overline\eta'}((\Cl_{\overline G}(\overline x)\cap\overline H)^+).
		\addtocounter{thm}{1}\tag{\thethm}
	\end{equation}
	
	First we assume that $\{yN\mid y\in\Cl_G(x)\}\subseteq NH/N$, that is, $\Cl_{\overline G}(\overline x)\subseteq\overline H$. 
	Without loss of generality, we assume that $y_1'\in H\cap xN$ (that is, $xN=y_1'N$).
	By assumption, \[\la_{\ttheta_x}((\Cl_G(x)\cap xN)^+)=\la_{\tvhi_{y_1'}}((\Cl_G(x)\cap y_1'M)^+)\]  and thus we conclude from Equations (\ref{equ:cen-char-ttheta-eta}) and (\ref{equ:cen-char-tvhi-eta}) that (\ref{equ:central-char-clifford}) holds.
	
	Now assume that $\{yN\mid y\in\Cl_G(x)\}\nsubseteq NH/N$.
	Then there exists $y\in G\setminus NH$ with $yN\notin NH/N$. In particular, $\Cl_G(x)\cap yN\cap H=\emptyset$.
	By assumption, \[\la_{\ttheta_y}((\Cl_G(y)\cap yN)^+)=\la_{\Res^H_{\langle N,y\rangle\cap H}(\tvhi)}((\Cl_G(x)\cap yN\cap H)^+)=0.\]
	According to Lemma~\ref{lem:central-char}(b), $\la_{\ttheta\eta}(\Cl_G(x)^+)=\la_{\ttheta\eta}(\Cl_G(y)^+)=0$ and by Lemma~\ref{lem:central-char}(a), 
	\[\la_{\ttheta_g}((\Cl_G(x)\cap gN)^+)=\la_{\ttheta_y}((\Cl_G(x)\cap yN)^+)=0\] for every $g\in\Cl_G(y)$.
	Therefore, by assumption,
	\[\la_{\tvhi_{y_{i}'}}((\Cl_G(x)\cap y_i'M)^+)=\la_{\ttheta_{y_i'}}((\Cl_G(x)\cap y_i'N)^+)=0\] for $1\le i\le s$.
	Then (\ref{equ:cen-char-tvhi-eta}) implies that $\la_{\tvhi\eta'}((\Cl_G(x)\cap H)^+)=0$ and this completes the proof.
\end{proof}

\subsection{Dade's ramification group}

Let $N\unlhd G$ and $b$ be a block of $N$.
We will make use of \emph{Dade's ramification group} $G[b]$, which was introduced in \cite{Da73}.
We recall the reformulation of the definition from \cite{Mu13}; the group $G[b]$ is generated by $N$ and all the elements $g\in G$ with $\la_{\widetilde b^{(g)}}(\Cl_{\langle N, g\rangle}(g)^+)\ne0$ for some block $\widetilde b^{(g)}$ of $\langle N, g\rangle$ covering $b$.
In particular, $G[b]\unlhd G_b$ and $G[b]\le N\C_G(D)$ where $D$ is a defect group of $b$.
See \cite{KS15,Mu13,Sp17} for further properties of $G[b]$.

\begin{lem}\label{lem:induced-blocks2}
	Let $N\le K\unlhd G$ and $H\le G$ with $N\unlhd G$ and $M=N\cap H$.
	Suppose that both $G/K$ and $K/N$ are abelian, and there exist characters $\ttheta\in\Irr(G)$ and $\tvhi\in\Irr(H)$ with \[\theta:=\Res^G_{N}(\ttheta)\in\Irr(N)\ \ \textrm{and}\ \ \vhi:=\Res^H_M(\tvhi)\in\Irr(M).\]
	Assume that a defect group $D$ of $\bl(\vhi)$ satisfies that $\C_G(D)\le H$.
	If \[\bl(\Res^G_K(\ttheta))=\bl(\Res^H_{H\cap K}(\tvhi))^K,\] then there exists an extension $\ttheta'$ of $\Res^G_K(\ttheta)$ to $G$ such that \[\bl(\Res^G_J(\ttheta'))=\bl(\Res^H_{J\cap H}(\tvhi))^J\] for every $N\le J\le G$.
\end{lem}	

\begin{proof}
	Let $b=\bl(\theta)$ and $b'=\bl(\vhi)$.
	Since a defect group $D$ of $b'$ satisfies that $\C_G(D)\le H$, for every $N\le J\le G$ and every block $\tilde b'$ of $J\cap H$ covering $b'$, some defect group $\tD$ of $\tilde b'$ satisfies $D\le\tD$ and hence $\C_J(\tD)\le\C_J(D)\le J\cap H$; in particular, $\tilde b'^J$ is defined.
	By \cite[Prop.~2.5(a)]{Sp17}, \[G[b]\le N\C_G(\tD)\le N\C_G(D)\le NH.\]
	
	According to Lemma~\ref{lemma:blockinductions1}, \[\bl(\Res^G_J(\ttheta))=\bl(\Res^H_{J\cap H}(\tvhi))^J\] for every $N\le J\le K$.
	By \cite[Thm.~C]{KS15}, there exists $\widetilde \phi\in\Irr(G[b])$ such that $\Res^{G[b]}_N(\widetilde \phi)$ is irreducible and 
\[\bl(\Res^{G[b]}_{J}(\widetilde\phi))=\bl(\Res^{H}_{J\cap H}(\tvhi))^{J}\] for every $N\le J\le G[b]$.
	Thus \[bl(\Res^G_J(\ttheta))=\bl(\Res^{G[b]}_{J}(\widetilde\phi))\] for every $N\le J\le K[b]$, and then we conclude from \cite[Lemma~3.2]{CS15} that there exists an extension $\ttheta'$ of $\Res^G_K(\ttheta)$ to $G$ such that \[\bl(\Res^G_J(\ttheta'))=\bl(\Res^{G[b]}_{J}(\widetilde\phi))\] for every
	$N\le J\le G[b]$.
	Therefore, \[\bl(\Res^G_J(\ttheta'))=\bl(\Res^{H}_{J\cap H}(\tvhi))^{J}\] for every $N\le J\le G[b]$.
	
	Now let $N\le J\le G$ and $J'=J\cap G[b]$. Following \cite[Thm.~3.5]{Mu13}, \[\bl(\Res^G_J(\ttheta'))=\bl(\Res^G_{J'}(\ttheta'))^J\] is the unique block of $J$ covering $\bl(\Res^G_{J'}(\ttheta'))$.
	Since $\bl(\Res^G_{J'}(\ttheta'))=\bl(\Res^{H}_{J'\cap H}(\tvhi))^{J}$, we have \[\bl(\Res^G_{J}(\ttheta'))=\bl(\Res^{H}_{J\cap H}(\tvhi))^{J}\] by \cite[Lemma~2.3]{KS15}.
\end{proof}

\section{A new relation between character triples}\label{sec:new-relation}

Let $G$ be a finite group, $N\unlhd G$ and $\theta\in\Irr(N)$.
We denote by $\Char(G\mid\theta)$ the set of characters of $G$ whose
constituents are in $\Irr(G\mid\theta)$. If $\theta$ is $G$-invariant, then
$(G,N,\theta)$ is called a \emph{character triple}.

\subsection{Character triples and projective representations}

A \emph{projective $\CC$-representation} of $G$ is a
function $\cP\colon G\to\GL_n(\CC)$ such that for every $g,h\in G$ there exists
some $\al(g,h)\in\CC^\ti$ satisfying $\cP(g)\cP(h)=\al(g,h)\cP(gh)$.
Then $\al$ is a $2$-cocycle and called the \emph{factor set} of $\cP$.
For a subset $C\subseteq G$, we define $\cP(C^+)=\sum\limits_{g\in C}\cP(g)$.

Let $(G,N,\theta)$ be a character triple. We say that a projective
representation $\cP\colon G\to\GL_{\theta(1)}(\CC)$ is \emph{associated with}
$(G,N,\theta)$ if $\cP\rceil_N$ is a linear representation of $N$ affording
$\theta$, and $\cP(ng)=\cP(n)\cP(g)$ and $\cP(gn)=\cP(g)\cP(n)$ for every
$n\in N$ and $g\in G$.
The factor set $\al$ of $\cP$ satisfies that $\al(1,1)=\al(g,n)=\al(n,g)=1$ for
every $n\in N$ and $g\in G$, and thus it can be regarded as a $2$-cocycle
of $G/N$. According to \cite[11.2]{Is06} there exists a projective
representation to every character triple.

\begin{rmk}
	Let $G=NH$ with $N\unlhd G$, and let $M=N\cap H$ and $(H,M,\theta)$ be a character triple.
	Suppose that $\cP\colon H\to\GL_m(\CC)$ is a projective $\CC$-representation of~$H$ associated with $(H,M,\theta)$.
	Let $n_1,\ldots, n_s\in N$ (with $s=|G|/|H|$) be representatives of the $M$-cosets in $N$.
	Then $n_1,\ldots, n_s$ are representatives of the $H$-cosets in $G$.
	For $1\le i,j\le s$, let $\widehat\cP_{i,j}\colon G\to \GL_{m}(\CC)$ be the map given by 
	\[\widehat\cP_{i,j}(x) =
	\begin{cases}
		\cP(n_i^{-1}xn_j)  & \ \mbox{if $n_i^{-1}xn_j\in H$,} \\
		0  &  \ \mbox{otherwise.}
	\end{cases}   \] 
	As in \cite[p.~712]{NS14}, we define the \emph{induced projective $\CC$-representation of~$\cP$ to~$G$} (with respect to $n_1,\ldots, n_s$) to be
	\[\Ind^G_{H,N}(\cP)\colon G\to\GL_{ms}(\CC),\ x\mapsto
	\left (
	\begin{matrix}
		\widehat\cP_{1,1}(x) & \cdots & \widehat\cP_{1,s}(x) \\
		\vdots &  & \vdots \\
		\widehat\cP_{s,1}(x) & \cdots & \widehat\cP_{s,s}(x)
	\end{matrix}
	\right  ).\]
	Straightforward computations show that $\Ind^G_{H,N}(\cP)$ is a projective $\CC$-representation of $G$ with factor set $\widehat \al$ such that for every $x,y\in G$, $\widehat \al(x,y)=\al(x_1,y_1)$ where $x_1,y_1\in H$ satisfy $x=x_1 n_x$ and $y=y_1n_y$ with $n_x,n_y\in N$.
	In other words, $\widehat\al$ and $\al$ coincide via the isomorphism $G/N\cong H/M$.
	If $\Ind_M^N(\theta)\in\Irr(N)$, then $\Ind^G_{H,N}(\cP)$ is a projective $\CC$-representation associated with the character triple $(G,N,\Ind_M^N(\theta))$. 
	
	In particular, if $\cP$ is a linear $\CC$-representation of~$H$, then $\Ind^G_{H,N}(\cP)$ is the induced linear $\CC$-representation of~$\cP$ to~$G$ (see for example \cite[p.~170]{NT89}).
\end{rmk}

\subsection{A generalization of central and block isomorphisms}

Now we give a generalization of central and block isomorphisms from
\cite{NS14,Sp18}.

\begin{defn}\label{def-1-char-triple} 
Suppose that $(G,N,\theta)$ and $(H,M,\vhi)$ are two character triples. 
We write
\[(G,N,\theta)\geqslant_{(g)}(H,M,\vhi)\]
if the following conditions are satisfied.
\begin{enumerate}[\rm(1)]
\item $H\le G$ and $N\cap H=M$.
\item There exists projective $\CC$-representations $\cP$ and $\cP'$ associated
 with $(G,N,\theta)$ and $(H,M,\vhi)$ respectively, with factor sets $\al$ and
 $\al'$ such that $\al\rceil_{H\ti H}=\al'$.
\end{enumerate}
In this situation, we say that \emph{$(G,N,\theta)\geqslant_{(g)}(H,M,\vhi)$ is
given by $(\cP,\cP')$}.
\end{defn}

\begin{thm}   \label{thm:iso-char-triple}	
 Suppose that $(G,N,\theta)\geqslant_{(g)}(H,M,\vhi)$ is given by $(\cP,\cP')$.
 Let~$\al$ be the factor set of~$\cP$. Then for every intermediate subgroup
 $N\le J\le G$, the linear map
 \[\sigma_J\colon\Char(J\mid \theta)\to\Char(J\cap H\mid\vhi)\]
 given by
 \[\trace(\cQ\otimes\cP\rceil_J)\mapsto
   \trace(\cQ\rceil_{J\cap H}\otimes\cP'\rceil_{J\cap H}),\]
 for any projective representation $\cQ$ of $J/N (\ge(J\cap H)/M)$ whose factor
 set is inverse to the one of $\cP\rceil_{J}$, is a well-defined map satisfying
 that for any $N\le K\le J\le G$ and any $\chi\in\Char(J\mid\theta)$,
 \begin{enumerate}[\rm(a)]
	\item $\Res^{J\cap H}_{K\cap H}(\sigma_{J}(\chi))
     =\sigma_{K}(\Res^{J}_{K}\chi)$, 
	\item $\sigma_{J}(\chi\beta)=\sigma_{J}(\chi)\Res^J_{J\cap H}(\beta)$ for
     every $\beta\in\Char(J/N)$, and
	\item $\sigma_{J}(\chi)^h=\sigma_{J^h}(\chi^h)$ for every $h\in H$.
	\item If $J=N(J\cap H)$, then $\sigma_J$ is a bijection.
 \end{enumerate}
 In particular, $\sigma_J$ induces a strong isomorphism between the character
 triples $(NH,N,\theta)$ and $(H,M,\vhi)$ in the sense of
 \cite[Problem~11.13]{Is06}.
 \begin{enumerate}[\rm(e)]
	\item[\rm(e)] If $J/N$ is abelian and $\theta$ extends to $J$, then
     $\sigma_J$ gives a surjection $\Irr(J\mid\theta)\to\Irr(J\cap H\mid\vhi)$.
 \end{enumerate}
\end{thm}

\begin{proof}
(a)--(d) follow by similar arguments as in the proof of Theorem~2.2 and Corollary~2.4 of
\cite{Sp18}.
Property (e) follows by the fact that every irreducible projective
representation of $(J\cap H)/M$ whose factor set is inverse to the one of
$\cP\rceil_{J\cap H}$ is a restriction of some irreducible projective
representation of $J/N$ whose factor set is inverse to the one of
$\cP\rceil_{J}$, when $J/N$ is abelian and $\theta$ extends to $J$.
\end{proof}

\begin{lem}   \label{lem:iso-char-triples1}
 Suppose that $(G,N,\theta)\geqslant_{(g)}(H,M,\vhi)$ is given by $(\cP,\cP')$.
 Let $\sigma$ be the linear maps defined by $(\cP,\cP')$ as in
 Theorem~\ref{thm:iso-char-triple}. We assume further that there exists a
 normal subgroup $G_0$ of $G$ such that $N\le G_0$, $G=G_0H$ and $G_0/N$ is
 abelian. Suppose that $\psi_0$ is an extension of $\theta$ to $G_0$.
 Then for every intermediate subgroup $G_0\le J\le G_{\psi_0}$, $\sigma_J$
 gives a bijection \[\Irr(J\mid\psi_0)\to \Irr(J\cap H\mid\sigma_{G_0}(\psi_0)).\]
\end{lem}

\begin{proof}	
Let $\al$ and $\al'$ be the factor sets of the projective representation $\cP$
and $\cP'$, respectively. First, the character $\psi_0$ is afforded by
$\cQ\otimes\cP\rceil_{G_0}$ where $\cQ$ is a projective $\CC$-representation of
$G_0/N$. Let $\cD$ be a projective $\CC$-representation of $G_{\psi_0}$
associated with the character triple $(G_{\psi_0},G_0,\psi_0)$ with
$\cD\rceil_{G_0}=\cQ\otimes\cP\rceil_{G_0}$.
Suppose that $\cD$ has the factor set $\beta$. Using the proof of
\cite[Thm.~8.16]{Na98}, we can write $\cD=\cU\otimes\cP\rceil_{G_{\psi_0}}$
where $\cU$ is a projective $\CC$-representation of $G_{\psi_0}$.
This implies that $\cU$ has factor set
$\beta(\al\rceil_{G_{\psi_0}\ti G_{\psi_0}})^{-1}$ and $\cU\rceil_{G_0}=\cQ$.

Let $\psi_0'=\sigma_{G_0}(\psi_0)$. Then $\psi'_0\in\Irr(G_0\cap H\mid\vhi)$ by
Theorem~\ref{thm:iso-char-triple}(e).
Define $\cD'=\cU\rceil_{H_{\psi_0}}\otimes\cP'\rceil_{H_{\psi_0}}$, which is a
projective $\CC$-representation of $H_{\psi_0}$ associated with
$(H_{\psi_0}, G_0\cap H,\psi'_0)$. 
It can be checked directly that $(\cD,\cD')$ gives
\[(G_{\psi_0},G_0,\psi_0)\geqslant_{(g)}(H_{\psi_0},G_0\cap H,\psi'_0).\]
Note that $G_{\psi_0}=G_0 H_{\psi_0}$. Let $\sigma'$ be the isomorphism between
the character triples $(G_{\psi_0},G_0,\psi_0)$ and
$(H_{\psi_0},G_0\cap H,\psi'_0)$. We can show that
$\sigma'_J(\psi)=\sigma_J(\psi)$ for $G_0\le J\le G_{\psi_0}$ and
$\psi\in\Irr(J\mid\psi_0)$. Thus we complete the proof.
\end{proof}

Let $\cP$ be a projective representation associated with a character triple
$(G,N,\theta)$. Then the matrix $\cP(c)$ is scalar for all $c\in\C_G(N)$.

\begin{defn}\label{def-central-isomrophism} 
 Suppose that $(G,N,\theta)\geqslant_{(g)}(H,M,\vhi)$ is given by $(\cP,\cP')$.
 We write
 \[(G,N,\theta)\geqslant_{(g),c}(H,M,\vhi)\]
 if the following conditions are satisfied:
 \begin{enumerate}[\rm(1)]
	\item $\C_G(N)\le H$.
	\item For every $c\in\C_G(N)$ the scalars associated to $\cP(c)$ and
     $\cP'(c)$ coincide.
 \end{enumerate}
 In this situation, we say that \emph{$(G,N,\theta)\geqslant_{(g),c}(H,M,\vhi)$
 is given by $(\cP,\cP')$}.
\end{defn}

\begin{defn} \label{def-block-isomrophism} 
 Suppose that $(G,N,\theta)\geqslant_{(g),c}(H,M,\vhi)$ is given
 by $(\cP,\cP')$. We write
 \[(G,N,\theta)\geqslant_{(g),b}(H,M,\vhi)\]
 if the following conditions are satisfied:
 \begin{enumerate}[\rm(1)]
	\item A defect group $D$ of $\bl(\vhi)$ satisfies $\C_G(D)\le H$.
	\item The maps $\sigma_J$ ($N\le J\le G$) induced by $(\cP,\cP')$ as in
     Theorem~\ref{thm:iso-char-triple} satisfy that
     \[\bl(\psi)=\bl(\sigma_J(\psi))^J\]
     for every subgroup $N\le J\le G$ and every $\psi\in\Irr(J\mid \theta)$
     with $\sigma_J(\psi)\in\Irr(J\cap H\mid\vhi)$.		
 \end{enumerate}
 In this situation, we say that \emph{$(G,N,\theta)\geqslant_{(g),b}(H,M,\vhi)$
 is given by $(\cP,\cP')$}.
\end{defn}

In Definition~\ref{def-block-isomrophism}, for every $N\le J\le G$ and every
character $\tvhi\in\Irr(J\cap H\mid\vhi)$, some defect group $\tD$ of
$\bl(\tvhi)$ satisfies $D\le\tD$ and hence
$\C_J(\tD)\le\C_J(D)\le J\cap H$; in particular, $\bl(\tvhi)^J$ is
defined.

\begin{lem}   \label{lem:ectend-block-iso}
 Let $N\unlhd G$, $H\le G$ and $M=H\cap N$. Let $\ttheta\in\Irr(G)$ and
 $\tvhi\in\Irr(H)$ with $\theta:=\Res^G_{N}(\ttheta)\in\Irr(N)$ and
 $\vhi:=\Res^H_M(\tvhi)\in\Irr(M)$.
 Assume that a defect group $D$ of $\bl(\vhi)$ satisfies $\C_G(D)\le H$.
 Suppose that the linear $\CC$-representations $\cP$ and $\cP'$ of $G$ and $H$
 afford $\ttheta$ and $\tvhi$ respectively.
 Then the following are equivalent. 
 \begin{enumerate}[\rm(a)]
  \item	$(\cP,\cP')$ gives $(G,N,\theta)\geqslant_{(g),b}(H,M,\vhi)$.
  \item \begin{enumerate}[\rm(1)]	
	\item $\Irr(\C_G(N)\mid\theta)=\Irr(\C_G(N)\mid\vhi)$, and
	\item $\bl(\Res^G_J(\ttheta))=\bl(\Res^H_{J\cap H}(\tvhi))^J$ for every
     subgroup $N\le J\le G$ with $J/N$ cyclic.
  \end{enumerate}
 \end{enumerate}
\end{lem}

\begin{proof}
By the definition, (a) implies (b). Now we prove that (b) implies (a).	
The projective representations $\cP$ and $\cP'$ are linear and hence their
factor sets are trivial.
Thus the conditions of Definition~\ref{def-1-char-triple} are satisfied.
By (b), $\bl(\vhi)^N=\bl(\theta)$, and thus $\bl(\theta)$ has a defect group
$\tD$ with $D\le\tD$.
We deduce from this that \[\C_G(N)\le\C_G(\tD)\le\C_G(D)\le H.\] 
By (b.1), the conditions of Definition~\ref{def-central-isomrophism} are
satisfied. Let us denote by $\sigma$ the linear maps defined by $(\cP,\cP')$
as in Theorem~\ref{thm:iso-char-triple}.
Then $\sigma_J(\Res^G_J(\ttheta)\eta)
 =\Res^H_{J\cap H}(\tvhi)\Res^J_{J\cap H}(\eta)$
for every $N\le J\le G$ and $\eta\in\Irr(J/N)$ with
$\Res^J_{J\cap H}(\eta)\in\Irr(J\cap H)$.
Therefore, the conditions of Definition~\ref{def-block-isomrophism} follow from Proposition~\ref{prop:induced-extensions}.
\end{proof}

\begin{defn}[Definitions~2.1,~2.7 and~4.2 of~\cite{Sp18}] 
 Let $*\in\{\,\emptyset,c,b\,\}$. Suppose that
 \[(G,N,\theta)\geqslant_{(g),*}(H,M,\vhi)\]
 is given by $(\cP,\cP')$. If $G=NH$, then we say that \emph{$(\cP,\cP')$ gives
 \[(G,N,\theta)\geqslant_{*}(H,M,\vhi).\]}

 Following \cite{NS14,Sp18}, the relation ``$\geqslant_c$'' (resp.\ 
 ``$\geqslant_b$'') is called \emph{central} (resp. \emph{block})
 \emph{isomorphism} between character triples.
\end{defn}

\begin{rmk}\label{rmk:gen-isomorphisms} 
 Suppose that $(G,N,\theta)$ and $(H,M,\vhi)$ are two character triples with
 $H\le G$ and $M=H\cap N$. Let $\cP$ and $\cP'$ be projective
 $\CC$-representations associated $(G,N,\theta)$ and $(H,M,\vhi)$
 respectively. 	
 \begin{enumerate}[\rm(a)]
  \item $(\cP,\cP')$ gives $(G,N,\theta)\geqslant_{(g)}(H,M,\vhi)$ if and only
   if $(\cP\rceil_{NH},\cP')$ gives $(NH,N,\theta)\geqslant(H,M,\vhi)$.
  \item Assume further $\C_G(N)\le H$. Then $(\cP,\cP')$ gives
   $(G,N,\theta)\geqslant_{(g),c}(H,M,\vhi)$ if and only if
   $(\cP\rceil_{NH},\cP')$ gives $(NH,N,\theta)\geqslant_c(H,M,\vhi)$.
  \item A similar statement holds for $\geqslant_{(g),b}$; see Corollary~\ref{rmk:gen-isomorphisms-b}. 
 \end{enumerate}
\end{rmk}	

\begin{rmk}
If $(G,N,\theta)\geqslant_{(g),*}(H,M,\vhi)$, then
$(J,N,\theta)\geqslant_{(g),*}(J\cap H,M,\vhi)$ for $*\in\{\,\emptyset,c,b\,\}$
and every $N\le J\le G$.
\end{rmk}

\begin{lem}\label{lem:change-proj}
 Suppose that $(G,N,\theta)\geqslant_{(g),*}(H,M,\vhi)$ is given by
 $(\cP,\cP')$ where $*\in\{\,\emptyset,c,b\,\}$.
 Then for any projective $\CC$-representation $\cP_1$ associated with
 $(G,N,\theta)$, there exists a projective $\CC$-representation $\cP'_1$
 associated with $(H,M,\vhi)$ such that $(\cP_1,\cP_1')$ also gives
 $(G,N,\theta)\geqslant_{(g),*}(H,M,\vhi)$ and the linear maps between the
 character triples $(G,N,\theta)$ and $(H,M,\vhi)$ defined as in
 Theorem~\ref{thm:iso-char-triple} by $(\cP,\cP')$ and by $(\cP_1,\cP_1')$
 coincide.
\end{lem}	

\begin{proof}
By \cite[Lemma~10.10(b)]{Na18}, there exists a map $\mu\colon G/N\to\CC^\ti$
and a matrix $A\in\GL_{\vhi(1)}(\CC)$ such that
$\cP_1(g)=\mu(g) A\cP(g)A^{-1}$ for all $g\in G$.
By \cite[Prop.~2.5]{Sp18}, $(\cP_1, \cP_1')$ also gives
$(G,N,\theta)\geqslant_{(g)} (H,M,\vhi)$, where $\cP_1'=\mu\rceil_H \cP'$, and
the linear maps between the character triples $(G,N,\theta)$ and $(H,M,\vhi)$
defined by $(\cP,\cP')$ and by $(\cP_1,\cP_1')$ (as in
Theorem~\ref{thm:iso-char-triple}) coincide. Therefore, we can check that
$(\cP_1,\cP_1')$ also gives $(G,N,\theta)\geqslant_{(g),*}(H,M,\vhi)$.
\end{proof}	

\begin{lem}   \label{lem-partial-order}
 \begin{enumerate}[\rm(a)]
  \item Let $*\in\{\,\emptyset,c\}$. If
   $(G,N,\theta)\geqslant_{(g),*}(H,M,\vhi)$ and $(H,M,\vhi)\geqslant_{(g),*}
    (T,L,\ze)$, then $(G,N,\theta)\geqslant_{(g),*} (T,L,\ze)$.
  \item Suppose that $(G,N,\theta)\geqslant_{(g),b}(H,M,\vhi)$ and
   $(H,M,\vhi)\geqslant_{(g),b} (T,L,\ze)$. Assume that a defect group $D$ of
   $\bl(\ze)$ satisfies that $\C_G(D)\le T$. Then
   $(G,N,\theta)\geqslant_{(g),b} (T,L,\ze)$.
 \end{enumerate}
\end{lem}

\begin{proof}
Suppose that $(G,N,\theta)\geqslant_{(g),*}(H,M,\vhi)$ is given by $(\cP,\cP')$
and $(H,M,\vhi)\geqslant_{(g),*} (T,L,\ze)$ is given by $(\cP',\cP'')$.
The existence of such projective representations is implied by
Lemma~\ref{lem:change-proj}. From this, it can be checked that $(\cP, \cP'')$
gives $(G,N,\theta)\geqslant_{(g),*} (T,L,\ze)$ (similar as in the proof of
\cite[Lemma~3.8]{NS14}) and then this assertion holds.
\end{proof}

In \cite[Lemma~3.13]{NS14} and \cite[Prop.~3.13]{Sp17}, the behavior of block
isomorphic character triples is analyzed with respect to certain quotients.
Now we prove the following adaptation for $\geqslant_{(g),*}$.

\begin{prop}   \label{prop:going-to-central-quotient}
 Let $*\in\{c,b\}$.	Suppose that $(G,N,\theta)\geqslant_{(g),*}(H,M,\vhi)$.
 Let $Z\unlhd G$ with $Z\le\C_G(N)$ and $Z\cap N=1$.
 Let $\overline G=G/Z$, $\overline N=NZ/Z$, $\overline H=H/Z$
 and $\overline M=MZ/Z$, and let $\overline\theta\in\Irr(\overline N)$ and
 $\overline\vhi\in\Irr(\overline M)$ be the characters corresponding to
 $\theta$ and $\vhi$ respectively. 
 \begin{enumerate}[\rm(a)]
  \item Assume that either $Z$ is a $p'$-subgroup or a central
   $p$-subgroup of $G$. Then \[(\overline G,\overline N,\overline \theta)
    \geqslant_{(g),*}(\overline H,\overline M,\overline \vhi).\]
  \item Assume that $Z\le\Z(G)$ is central. Then $(\overline G,\overline N,\overline \theta)\geqslant_{(g),*}(\overline H,\overline M,\overline \vhi)$.
 \end{enumerate}
\end{prop}

\begin{proof}
Suppose that $(G,N,\theta)\geqslant_{(g),*}(H,M,\vhi)$ is given by $(\cP,\cP')$.
Let $\al$ and $\al'$ be the factor sets of $\cP$ and $\cP'$ respectively and
let us denote by $\sigma$ the linear maps defined by $(\cP,\cP')$ as in
Theorem~\ref{thm:iso-char-triple}.
Let $\overline{\cX}$ be the linear $\CC$-representation of $\overline N$
defined by $\overline{\cX}(nZ)=\cP(n)$ for $n\in N$ so that $\overline{\cX}$
affords $\overline\theta$.
There exists a projective $\CC$-representation $\overline\cP$ of $\overline G$
associated with $(\overline G,\overline N,\overline \theta)$ such that
$\overline\cP(nZ)=\overline{\cX}(nZ)=\cP(n)$ for $n\in N$.
The map $\cD$ on $G$ defined by $\cD(g)=\overline\cP(gZ)$ for every $g\in G$ is
a projective $\CC$-representation with $\cD\rceil_N=\cP\rceil_N$.
So there exists a map $\xi\colon G/N\to\CC^\ti$ such that
$\overline\cP(gZ)=\xi(gN)\cP(g)$ for every $g\in G$.

The map $\xi\rceil_H\cP'$ is a projective representation of $H$ associated to
$(H,M,\vhi)$ as well. It can be checked that $(\xi\cP,\xi\rceil_H\cP')$ gives
$(G,N,\theta)\geqslant_{(g)}(H,M,\vhi)$ and the linear maps defined by
$(\cP,\cP')$ and by $(\xi\cP,\xi\rceil_H\cP')$ coincide as in the proof of
Lemma~\ref{lem:change-proj}. As $\cP(c)$ and $\cP'(c)$ are associated with the
same scalar for $c\in\C_G(N)$, the projective representations $\xi\cP$ and
$\xi\rceil_H\cP'$ have the same property. From this $(\xi\cP,\xi\rceil_H\cP')$
also gives $(G,N,\theta)\geqslant_{(g),*}(H,M,\vhi)$.
Thus without loss of generality, we assume that $\xi=1_G$.
Therefore, $\al(gx,g'x')=\al(g,g')$ for every $g,g'\in G$ and $x,x'\in NZ$.
Taking into account that $\al\rceil_{H\ti H}=\al'$, this implies that $\cP'$
defines uniquely a projective $\CC$-representation $\overline\cP'$ of
$\overline H$ as well. Since $Z\unlhd G$ and $N\cap Z=1$, we get that
$\C_{\overline G}(\overline N)=\C_G(N)/Z$. So $(\overline\cP,\overline\cP')$
gives $(\overline G,\overline N,\overline\theta)\geqslant_{(g),c}(\overline H,\overline M,\overline\vhi)$.

From now on we assume that $*=b$.
Denote by $\overline\sigma$ the linear maps defined by
$(\overline\cP,\overline\cP')$ as in Theorem~\ref{thm:iso-char-triple}.
Let $NZ\le J\le G$, $\overline J:=J/Z$ and $\overline\psi\in\Char(\overline J\mid\overline\theta)$.
Then by the definition of $\sigma$ and $\overline\sigma$ the character
$\overline\sigma_{\overline J}(\overline\psi)$ lifts to $\sigma_J(\psi)$, where
$\psi\in\Char(J\mid\theta)$ is the inflation of $\overline\psi$.
Suppose that $\overline\psi\in\Irr(\overline J\mid\overline\theta)$ and
$\overline\sigma_{\overline J}(\overline\psi)\in\Irr((J\cap H)/Z\mid\overline\vhi)$.
By assumption, $\bl(\psi)=\bl(\sigma_J(\psi))^J$.
Let $\overline B$ be the block of $\overline J$ contained in $\bl(\psi)$ and
$\overline b$ the block of $(J\cap H)/Z$ contained in $\bl(\sigma_J(\psi))$,
in the sense of \cite[p.~198]{Na98}. According to \cite[Prop.~2.4]{NS14} we
have $\overline b^{\overline J}=\overline B$.
Note that some defect group $D$ of $\bl(\psi)$ satisfies $\C_G(D)\le H$, and
from this $\overline D=DZ/Z$ is a defect group of $\overline b$ according to
Theorems~8.8 and~8.10 of Chapter~5 of \cite{NT89}.
Now $\C_{\overline G}(\overline D)=\C_G(D)/Z$ because $N\cap Z=1$ and
$Z\unlhd G$ and this gives $\C_{\overline G}(\overline D)\le\overline H$.
Since $\overline b=\bl(\overline\psi)$ and
$\overline B=\bl(\overline\sigma_{\overline J}(\overline\psi))$, this implies
that $(\overline\cP,\overline\cP')$ gives $(\overline G,\overline N,\overline \theta)\geqslant_{(g),b}(\overline H,\overline M,\overline \vhi)$.

To prove (b), we can apply part (a) to $G$ with respect to $Z_p$ and then
again to $G/Z_p$ with respect to $Z/Z_p$.
\end{proof}

\begin{defn}   \label{defn-normal-iso}
 Let $*\in\{\,c,b\,\}$. Suppose that $(G,N,\theta)\geqslant_{(g),*}(H,M,\vhi)$
 is given by $(\cP,\cP')$, and let $M\le H_0\unlhd H$.
 We say that \emph{$(G,N,\theta)\geqslant_{(g),*}(H,M,\vhi)$ is given by
 $(\cP,\cP')$ normally with respect to $H_0$} if for any linear character
 $\iota\in\Lin(H_0/M\C_{H_0}(N))$, there exists a projective
 $\CC$-representation $\cP''$ associated with $(H_\iota,M,\vhi)$ such that
 \begin{enumerate}[\rm(1)]
  \item $(\cP,\cP'')$ gives $(G,N,\theta)\geqslant_{(g),*}(H_\iota,M,\vhi)$, and
  \item $\cP''(h)=\iota(h)\cP'(h)$ for any $h\in H_0$.
\end{enumerate}

We say that \emph{$(G,N,\theta)\geqslant_{(g),*}(H,M,\vhi)$ is normal with
respect to $H_0$} if $(G,N,\theta)\geqslant_{(g),c}(H,M,\vhi)$ is given by some
projective representations $(\cP,\cP')$ normally with respect to $H_0$.
\end{defn}

\begin{rmk}
Suppose that $(G,N,\theta)\geqslant_{(g),*}(H,M,\vhi)$ is normal with respect
to $H_0$. Let $N\le G_1\le G$, $H_1=H\cap G_1$ such that $H_0\le H_1$. 
Then $(G_1,N,\theta)\geqslant_{(g),*}(H_1,M,\vhi)$ is also normal with respect
to $H_0$. 
\end{rmk}	

\subsection{A criterion for $\geqslant_{(g),b}$}
Now we study the relation $\geqslant_{(g),b}$ between character triples in more detail and derive a criterion that generalizes \cite[Thm.~4.3]{Sp17}.
The following lemma is similar to \cite[Lemma~4.2]{Sp17}.

\begin{lem}\label{lem:block-induced}
Let $N\unlhd G$, $H\le G$ and $M=N\cap H$.
Suppose that $\ttheta\in\Irr(G)$ and $\tvhi\in\Irr(H)$ with $\theta=\Res^G_N(\ttheta)\in\Irr(N)$ and $\vhi=\Res^H_M(\tvhi)\in\Irr(M)$ such that $\bl(\vhi)$ has a defect group $D$ with $\C_G(D)\le H$.
Let $\cP$ and $\cP'$ be linear $\CC$-representations of $G$ and $H$ affording $\ttheta$ and $\tvhi$ respectively.
\begin{enumerate}[\rm(a)]
\item For $x\in G$ and $J=\langle N,x\rangle$, $\cP(\Cl_J(x)^+)$ and $\cP'((\Cl_J(x)\cap H)^+)$ are scalar matrices associated to  elements of $\cO$. So there is some element in $\FF$ associated with $\cP(\Cl_J(x)^+)^*$ (or $\cP'((\Cl_J(x)\cap H)^+)^*$). Here, $\cO$ and $\FF$ are defined as in \S\ref{subsec:induced-blocks}.
\item If $x\in G\setminus G[\bl(\theta)]$ and $J=\langle N,x\rangle$, then $\cP(\Cl_J(x)^+)^*$ is the zero matrix.
\item If $x\in G\setminus G[\bl(\vhi)^N]$ and $J=\langle N,x\rangle$, then $\cP'((\Cl_J(x)\cap H)^+)^*$ is the zero matrix.
\item If $x\in G[\bl(\theta)]$, then for $J=\langle N,x\rangle$ the following statements are equivalent.
\begin{enumerate}[\rm(i)]
	\item $\bl(\Res^G_J(\ttheta))=\bl(\Res^H_{J\cap H}(\tvhi))^J$.
	\item The matrices $\cP(\Cl_J(y)^+)^*$ and $\cP'((\Cl_J(y)\cap H)^+)^*$ are associated with the same element in $\FF$ for every $y\in xN$.
\end{enumerate}
\end{enumerate}
\end{lem}

\begin{proof}
Let $x\in G$ and $J=\langle N,x\rangle$. Then $\Cl_J(x)\subseteq xN$ and thus $\Cl_J(x)\cap H$ is either empty or contained in $x'M$ for some $x'\in H\cap xN$.
Since $\cP\rceil_J$ is a linear representation of $J$, $\cP(\Cl_J(x)^+)$ is a scalar matrix associated with an algebraic integer (see e.g. \cite[\S3]{Is06}).
On the other hand, $\Cl_J(x)\cap H$ is a union of $(J\cap H)$-conjugacy classes of $J\cap H$, and then $\cP'((\Cl_J(x)\cap H)^+)$ is also a scalar matrix associated with an algebraic integer.
This proves (a).

Note that $\Res^G_J(\ttheta)$ is afforded by $\cP\rceil_J$.
Then $\la_{\Res^G_J(\ttheta)}(\Cl_J(x)^+)$ is the scalar associated with $\cP(\Cl_J(x)^+)^*$.
Analogously, $\la_{\Res^H_{J\cap H}(\tvhi)}((\Cl_J(x)\cap H)^+)$ is the scalar associated with $\cP'((\Cl_J(x)\cap H)^+)$.
Thus (b) and (c) follow by the definition of Dade's ramification groups (cf. \cite{Mu13} or \cite[Prop.~2.5(b)]{NS14}).

Assume now $x\in G[\bl(\theta)]$.
The equality $\bl(\Res^G_J(\ttheta))=\bl(\Res^H_{J\cap H}(\tvhi))^J$ implies that \[\la_{\Res^G_J(\ttheta)}(\Cl_J(y)^+)=\la_{\Res^H_{J\cap H}(\tvhi)}((\Cl_J(y)\cap H)^+)\] for every $y\in J$.
This shows that (i) implies (ii) in part (d).

Finally we prove that (ii) implies (i) in part (d).
Using the assumption (ii), if we let $x\in N$, then we get $\bl(\vhi)^N=\bl(\theta)$.
By \cite[Thm.~B]{KS15} there exists $\ze\in\Irr(J/N)$ such that \[\bl(\ze\Res^G_J(\ttheta))=\bl(\Res^H_{J\cap H}(\tvhi))^J.\] Since $J/N$ is cyclic, $\ze$ is a linear character.
So \[\ze(y')^*\la_{\Res^G_J(\ttheta)}(\Cl_J(y)^+)=\la_{\Res^H_{J\cap H}(\tvhi)}((\Cl_J(y)\cap H)^+)\] for every $y\in J$ and $y'\in yN$.
By \cite[Cor.~3.3]{Mu13}, there exists $y\in xN$ with $\la_{\Res^G_J(\ttheta)}(\Cl_J(y)^+)\ne 0$.
Then by the assumption (ii) the matrices $\cP(\Cl_J(y)^+)^*$ and $\cP'((\Cl_J(y)\cap H)^+)^*$ are nonzero and associated with the same element in $\FF$. 
Hence $\ze(y)^*=1$, which implies that $\ze\in\Lin_p(J/N)$.
Accordingly, $\bl(\ze\Res^G_J(\ttheta))=\bl(\Res^G_J(\ttheta))$ and this completes the proof.
\end{proof}

\begin{thm}\label{thm:criterion-geqslant_{(g),b}}
	Suppose that $(G,N,\theta)\geqslant_{(g),c}(H,M,\vhi)$ is given by $(\cP,\cP')$. 
	Assume further that a defect group $D$ of $\bl(\vhi)$ satisfies $\C_G(D)\le H$.
		Then the following two statements are equivalent.
	\begin{enumerate}[\rm(a)]
		\item For every $x\in G[\bl(\theta)]$ the matrices $\cP(\Cl_{\langle N,x\rangle}(x)^+)^*$ and $\cP'((\Cl_{\langle N,x\rangle}(x)\cap H)^+)^*$ are associated with the same scalar in $\FF$.
		\item $(\cP,\cP')$  gives $(G,N,\theta)\geqslant_{(g),b}(H,M,\vhi)$.
	\end{enumerate}
\end{thm}

\begin{proof}	
Let $\al$ and $\al'$ be the factor sets of $\cP$ and $\cP'$ respectively and let us denote by $\sigma$ the linear maps defined by $(\cP,\cP')$ as in Theorem~\ref{thm:iso-char-triple}.	

First we prove that (b) implies (a). We have $\bl(\theta)=\bl(\vhi)^{N}$ by (b).
Let $x\in G[b]$ and $J=\langle N,x\rangle$ and let $\ttheta\in\Irr(J\mid\theta)$.
Then $\ttheta$ is an extension of $\theta$ and $\tvhi:=\sigma_J(\ttheta)\in\Irr(J\cap H\mid\vhi)$ is an extension of $\vhi$ (by Theorem~\ref{thm:iso-char-triple}(e)) since $J/N$ is cyclic.
Let $\cQ$ be a projective $\CC$-representation of $J/N$ such that $\cQ\otimes\cP\rceil_J$ affords $\ttheta$.
Then $\tvhi$ is afforded by $\cQ\rceil_{J\cap H}\otimes\cP'\rceil_{J\cap H}$.
Note that $\cQ$ is a one-dimensional projective representation with $\cQ(x)^*\ne 0$.
Since $\Cl_J(x)\subseteq xN$, $\bl(\ttheta)=\bl(\tvhi)^J$ implies according to Lemma~\ref{lem:block-induced} that the matrices \[(\cQ\otimes\cP\rceil_J)(\Cl_J(x)^+)^*=\cQ(x)^*\cP(\Cl_J(x)^+)^*\] and \[(\cQ\rceil_{J\cap H}\otimes\cP'\rceil_{J\cap H})((\Cl_J(x)\cap H)^+)^*=\cQ(x)^*\cP'((\Cl_J(x)\cap H)^+)^*\] are scalar matrices associated with the same element in $\FF$.
This implies (a).

Now we prove that (a) implies (b).	
For the character triple $(G,N,\theta)$ we can define a group $\hG$ with a surjective group homomorphism $\eps\colon\hG\to G$. We first recall the construction (see e.g. \cite[\S5.3]{Na18}).
Fix a finite subgroup $Z$ of $\CC^\ti$ that contains all the values of the factor set $\al$.
Let $\hG=\{\,(g,z)\mid g\in G,z\in Z\,\}$ where multiplication is defined by \[(g_1,z_1)(g_2,z_2)=(g_1g_2,\al(g_1,g_2)z_1z_2)\]
for $g_1,g_2\in G$ and $z_1,z_2\in Z$.
Then $\eps\colon\hG\to G$ given by $(g,z)\mapsto g$ is an epimorphism with kernel $1\ti Z\le\Z(\hG)$, which we will identify with $Z$. 
The group $N_0=N\ti 1$ is a normal subgroup of $\hG$ that is isomorphic to $N$ via $\eps\rceil_{N_0}$.
The action of $\hG$ on $N_0$ coincides with the action of $G$ on $N$ via~$\eps$.
Let $\theta_0=\theta\circ\eps\rceil_{N_0}\in\Irr(N_0)$.
Define a map $\mathrm{rep}\colon G\to\hG$ by $g\mapsto (g,1)$. Then $\eps\circ\mathrm{rep}=\mathrm{id}_G$ and $\mathrm{rep}(1)=1$. In addition, $\mathrm{rep}(n)\in N_0$ and $\mathrm{rep}(ng)=\mathrm{rep}(n)\mathrm{rep}(g)$ for $n\in N$ and $g\in G$.

The map $\cD$ defined on $\hG$ by $\cD((g,z))=z\cP(g)$ for every $z\in Z$ and $g\in G$ is an irreducible linear $\CC$-representation of $\hG$.
In addition, $\cD(\mathrm{rep}(g))=\cP(g)$ for every $g\in G$.
Let $\ttheta_0\in\Irr(\hG\mid\theta_0)$ be the character afforded by $\cD$.
Then $\ttheta_0$ is an extension of $\theta_0$ to $\hG$ and the unique irreducible constituent $\nu$ of $\Res^{\hG}_{Z}(\ttheta_0)$ is faithful.

For any subgroup $U\le G$, let us write \[\widehat U:=\eps^{-1}(U)=\{(u,z)\mid u\in U,z\in Z\}\le\hG.\]
If $c\in\C_G(N)$, then $\cP(c)$ is a scalar matrix and from this $\cP(n^{-1}cn)=\cP(c)$ for $n\in N$.
Thus one get $\eps(\C_{\hG}(\widehat U))=\C_G(U)$ for $N\le U\le G$.

As in \cite[Thm.~4.1]{NS14}, $\widehat M=M_0\ti Z$ where $M_0:=N_0\cap\widehat{M}$.
Let $\vhi_0=\vhi\circ\eps\rceil_{M_0}$.
Since $\al\rceil_{H\ti H}=\al'$, the values of $\al'$ are also contained in $Z$.
Therefore, the subgroup $\widehat H$ of $\hG$ provides a central extension of $H$.
As above, the map $\cD'$ defined on $\widehat H$ by $\cD'((h,z))=z\cP'(h)$ for every $z\in Z$ and $h\in H$ is an irreducible linear $\CC$-representation $\cD'$ of $\widehat H$ with $\cD'(\mathrm{rep}(h))=\cP'(h)$ for $h\in H$, and $\cD'$ affords an extension $\tvhi_0\in\Irr(\widehat H)$ of $\vhi_0$.
In addition, $\Irr(Z\mid\tvhi_0)=\{\nu\}$.
Thus it can be checked that $(\cD,\cD')$ gives \[(\hG,N_0,\vhi_0)\geqslant_{(g),c}(\widehat H,M_0,\vhi_0).\]
Let us denote by $\widehat\sigma$ the linear maps defined by $(\cD,\cD')$ as in Theorem~\ref{thm:iso-char-triple}.	
Then for $N\le J\le G$ and $\psi\in\Char(J\mid\theta)$ we have $\sigma_J(\psi)\circ\eps=\widehat\sigma_{\widehat J}(\psi\circ\eps\rceil_{\widehat J}).$
By Proposition~\ref{prop:going-to-central-quotient} and its proof, to prove $(\cP,\cP')$  gives $(G,N,\theta)\geqslant_{(g),b}(H,M,\vhi)$, it suffices to show that $(\cD,\cD')$ gives \[(\hG,N_0,\vhi_0)\geqslant_{(g),b}(\widehat H,M_0,\vhi_0).\]

Let $D_0=M_0\cap\widehat D$.
Since the action of $\hG$ on $N_0$ coincide with the action of $G$ on $N$ via~$\eps$, one has $\C_{\widehat G}(D_0)\le\widehat H$.
So it suffices to verify condition (2) of Definition~\ref{def-block-isomrophism}, and then by Proposition~\ref{prop:induced-extensions}, 
it remains to show \[\bl(\Res^{\hG}_{J_0}(\ttheta_0))=\bl(\Res^{\widehat H}_{J_0\cap\widehat H}(\tvhi_0))^J\] for every $N_0\le J_0\le \hG$ with cyclic $J_0/N_0$.
If we prove that $\cD(\Cl_{J_0}(y_0)^+)^*$ and $\cD'((\Cl_{J_0}(y_0)\cap \widehat H)^+)^*$ are associated with the same scalar in $\FF$ for every $y_0\in\hG[\bl(\theta_0)]$ where $J_0=\langle N_0,y_0\rangle$, then by Lemma~\ref{lem:block-induced} this theorem follows.

By (a) for elements of $N$, we get $\bl(\theta)=\bl(\vhi)^N$. 
Let $x\in G$ and $J=\langle N,x\rangle$.
By Lemma~\ref{lem:block-induced}, $\cP(\Cl_J(x)^+)^*$ and $\cP'((\Cl_J(x)\cap H)^+)^*$ are associated with the same scalar in $\FF$. 
Define for the element $x$ the map \[\mathcal L_x\colon N\to N\quad \textrm{by}\quad n\mapsto x^{-1}n^{-1}xn.\]
Then $\mathcal L_x$ depends on the automorphism of $N$ induced by~$x$ and $\Cl_J(x)=\{xl\mid l\in\mathcal L_x(N)\}$.
Let $x_0=(x,1)$ and $J_0=\langle N_0,x_0\rangle$ and define $\mathcal L_{x_0}\colon N_0\to N_0$ by $n\mapsto y^{-1}n^{-1}yn$.
Then \[\Cl_{J_0}(x_0)=\{x_0l\mid l\in\mathcal L_{x_0}(N_0) \}.\]

Note that the action of $x_0$ on $N_0$ coincides with the one of $x$ on $N$.
Therefore, \[\cD(\Cl_{J_0}(x_0)^+)=\sum\limits_{l\in\mathcal L_{x_0}(N_0)}\cD(x_0l)=\sum\limits_{l\in\mathcal L_{x}(N)}\cP(xl)=\cP(\Cl_{J}(x)^+)\] and analogously (see also \cite[p.~1088]{Sp17}) \[\cD'((\Cl_{J_0}(x_0)\cap \widehat H)^+)=\cP'((\Cl_{J}(x)\cap H)^+).\]

Finally, let  $y_0=(x,z)$. Then $\Cl_{\langle N_0,y_0\rangle}(y_0)=\{\, k(1,z)\mid k\in \Cl_{J_0}(x_0) \,\}$.
From this, \[\cD(\Cl_{\langle N_0,y_0\rangle}(y_0)^+)=z\cP(\Cl_{J}(x)^+)\] and \[\cD'((\Cl_{\langle N_0,y_0\rangle}(y_0)\cap \widehat H)^+)=z\cP'((\Cl_{J}(x)\cap H)^+).\]
Thus
$\cD(\Cl_{\langle N_0,y_0\rangle}(y_0)^+)^*$ and $\cD'((\Cl_{\langle N_0,y_0\rangle}(y_0)\cap \widehat H)^+)^*$ are associated with the same element in $\FF$ and we complete the proof.
\end{proof}

\begin{cor}\label{rmk:gen-isomorphisms-b} 
	Suppose that $(G,N,\theta)$ and $(H,M,\vhi)$ are two character triples with $H\le G$ and $M=N\cap H$.
	Assume further that a defect group $D$ of $\bl(\vhi)$ satisfies $\C_G(D)\le H$.	
	Let $\cP$ and $\cP'$ be projective $\CC$-representations associated with  $(G,N,\theta)$ and $(H,M,\vhi)$ respectively. 
Then $(\cP,\cP')$ gives $(G,N,\theta)\geqslant_{(g),b}(H,M,\vhi)$ if and only if  $(\cP\rceil_{NH},\cP')$ gives
$(NH,N,\theta)\geqslant_b(H,M,\vhi)$.
\end{cor}	

\begin{proof}
We assume that $\bl(\vhi)^N=\bl(\theta)$ since this is implied by either $(G,N,\theta)\geqslant_{(g),b}(H,M,\vhi)$ or $(NH,N,\theta)\geqslant_b(H,M,\vhi)$.
Then there exists a defect group $Q$ of $\bl(\theta)$ with $D\le Q$, and accordingly \[\C_G(N)\le\C_G(Q)\le\C_G(D)\le H.\]
By \cite[Cor.~12.6]{Da73} (see also \cite[Prop.~2.5(a)]{Sp17}), $G[\bl(\theta)]\le N\C_G(Q)\le NH$.
Thus this assertion follows by Remark~\ref{rmk:gen-isomorphisms}(b) and Theorem~\ref{thm:criterion-geqslant_{(g),b}}.
\end{proof}

\section{Construction of $\geqslant_{(g),c}$ and $\geqslant_{(g),b}$}\label{sec:construction}

Next, we consider how we can obtain new pairs of character triples under the relation $\geqslant_{(g),c}$ and $\geqslant_{(g),b}$ by using direct products and wreath products. We also show that the relation between character triples only depends on the automorphisms induced. These statement generalizes the theorems in \cite[\S5]{Sp17}.

First we consider character triples coming from direct products.

\begin{thm}\label{thm:direct-prod}
\begin{enumerate}[\rm(a)]
\item	Suppose that $(G_i,N_i,\theta_i)\geqslant_{(g),*}(H_i,M_i,\vhi_i)$ for $i=1,2$ and $*\in\{c,b\}$.
	Then \[(G_1\ti G_2,N_1\ti N_2,\theta_1\ti\theta_2)\geqslant_{(g),*}(H_1\ti H_2,M_1\ti M_2,\vhi_1\ti\vhi_2).\]
\item If furthermore $(G_i,N_i,\theta_i)\geqslant_{(g),*}(H_i,M_i,\vhi_i)$ is normal with respect to $H_{i,0}$ for $i=1,2$ and $*\in\{c,b\}$, then 
 \[(G_1\ti G_2,N_1\ti N_2,\theta_1\ti\theta_2)\geqslant_{(g),*}(H_1\ti H_2,M_1\ti M_2,\vhi_1\ti\vhi_2)\] is normal with respect to $H_{1,0}\ti H_{2,0}$.
\end{enumerate}
\end{thm}

\begin{proof}
(a) 
First, the group-theoretical conditions (Definitions~\ref{def-1-char-triple}~(a), \ref{def-central-isomrophism}~(a) and~\ref{def-block-isomrophism}~(a)) hold by direct calculation.
Suppose that $(\cP_i,\cP_i')$ gives $(G_i,N_i,\theta_i)\geqslant_{(g),b}(H_i,M_i,\vhi_i)$, where $\cP_i$ and $\cP_i'$ are projective $\CC$-representations associated with $(G_i,N_i,\theta_i)$ and $(H_i,M_i,\vhi_i)$ respectively for $i=1,2$.
Let $\widetilde\cP$ and $\widetilde\cP'$ be defined by \[\widetilde\cP(g_1,g_2)=\cP_1(g_1)\otimes\cP_2(g_2)\quad \textrm{for}\quad g_i\in G_i,\] and \[\widetilde\cP'(g_1,g_2)=\cP'_1(h_1)\otimes\cP'_2(h_2)\quad \textrm{for}\quad h_i\in H_i.\]
Then $\widetilde\cP$ and $\widetilde\cP'$ are projective $\CC$-representation associated with $(G_1\ti G_2,N_1\ti N_2,\theta_1\ti\theta_2)$ and $(H_1\ti H_2,M_1\ti M_2,\vhi_1\ti\vhi_2)$ respectively.
By Remark~\ref{rmk:gen-isomorphisms} and Corollary~\ref{rmk:gen-isomorphisms-b}, it suffices to check that $(\widetilde\cP|_{H_1N_1\ti H_2N_2},\widetilde\cP')$ gives \[(N_1H_1\ti N_2H_2,N_1\ti N_2,\theta_1\ti\theta_2)\geqslant_{*}(H_1\ti H_2,M_1\ti M_2,\vhi_1\ti\vhi_2)\]
which was shown in the proofs of \cite[Prop.~2.18]{Sp18} and \cite[Thm.~5.1]{Sp17}.

(b) Let $\iota\in\Lin(H_0/(M_1\ti M_2)\C_{H_0}(N_1\ti N_2))$ where $H_0=H_{1,0}\ti H_{2,0}$ and
 $\iota=\iota_1\ti\iota_2$ with $\iota_i\in\Lin(H_{i,0}/M_i\C_{H_{i,0}}(N_i))$ for $i=1,2$.
In (a) we assume further that $(G_i,N_i,\theta_i)\geqslant_{(g),*}(H_i,M_i,\vhi_i)$ is given by $(\cP_i,\cP'_i)$ normally with respect to $H_{i,0}$.
Then for $i=1,2$, by the assumption, there exists a projective $\CC$-representation $\cP_{i}''$ associated with the character triple $((H_i)_{\iota_i},M_i,\vhi_i)$ such that $(\cP_i,\cP''_{i})$ gives \[(G_i,N_i,\theta_i)\geqslant_{(g),*}((H_i)_{\iota_i},M_i,\vhi_i),\] and $\cP''_{i}(h)=\iota_i(h)\cP_i'(h)$ for any $h\in H_{i,0}$.
By the proof of (a), 
$(\cP_1\otimes\cP_2,\cP''_{1}\otimes\cP''_{2})$ gives \[(G_1\ti G_2,N_1\ti N_2,\theta_1\ti \theta_2)\geqslant_{(g),*}((H_1\ti H_2)_{\iota},M_1\ti M_2,\vhi_1\ti\vhi_2),\] and \[\cP''_{1}(h_1)\otimes\cP''_{2}(h_2)=\iota_1(h_1)\iota_2(h_2)\cP_1'(h_1)\otimes\cP_2'(h_2)\] for any $h_i\in H_{i,0}$.
Thus we complete the proof.
\end{proof}

Next, we consider character triples using wreath products.

\begin{thm}\label{thm:wreath-product}
\begin{enumerate}[\rm(a)]
\item Let $(G,N,\theta)\geqslant_{(g),*}(H,M,\vhi)$ with non-trivial $N$ and $*\in\{c,b\}$, and let $n$ be a positive integer.
Denote $\tG=G^n$, $\tN=N^n$, $\tH=H^n$, $\tM=M^n$ and $\ttheta=\theta\ti\cdots\ti\theta\in\Irr(\tN)$, $\tvhi=\vhi\ti\cdots\ti\vhi\in\Irr(\tM)$.	Then \[(G\wr \fS_n,\tN,\ttheta)\geqslant_{(g),*}(H\wr\fS_n,\tM,\tvhi).\]
\item In (a), if we assume further that $(G,N,\theta)\geqslant_{(g),*}(H,M,\vhi)$ is normal with respect to $H_0$, then \[(G\wr \fS_n,\tN,\ttheta)\geqslant_{(g),*}(H\wr\fS_n,\tM,\tvhi).\] is normal with respect to $\tH_0:=H_0^n$.
\end{enumerate}
\end{thm}

\begin{proof}	
(a)	 
Suppose that $(\cP,\cP')$ gives $(G,N,\theta)\geqslant_{(g),b}(H,M,\vhi)$, where $\cP$ and $\cP'$ are projective $\CC$-representations associated with $(G,N,\theta)$ and $(H,M,\vhi)$ respectively.
Let \[\mathcal R\colon \fS_n\to\GL_{n\theta(1)}(\CC)\quad \textrm{and} \quad \mathcal R'\colon \fS_n\to\GL_{n\vhi(1)}(\CC)\] be the linear $\CC$-representations of the symmetric group $\fS_n$ defined as in \cite[Definition~2.20]{Sp18}.
Let $\widetilde\cP\colon G\wr \fS_n\to\GL_{\theta(1)^n}(\CC)$ be defined by \[ \widetilde\cP((g_1,\ldots,g_n)\sigma)=(\cP(g_1)\otimes\cdots\otimes\cP(g_n))\mathcal R(\sigma)\]
for $g_1,\ldots,g_n\in G$ and $\sigma\in\fS_n$ and $\widetilde\cP'\colon H\wr \fS_n\to\GL_{\theta(1)^n}(\CC)$ analogously.
Using the proofs of \cite[Thm.~2.21]{Sp18} and \cite[Thm.~5.2]{Sp17}, it can be checked that
\[(HN\wr \fS_n,\tN,\ttheta)\geqslant_{*}(H\wr\fS_n,\tM,\tvhi)\] is given by
 $(\widetilde\cP\rceil_{HN\wr \fS_n},\widetilde\cP')$.
Then the assertion holds by Remark~\ref{rmk:gen-isomorphisms} and Corollary~\ref{rmk:gen-isomorphisms-b}.

(b) Let $\iota\in\Lin(\tH_0/\tM\C_{\tH_0}(\tN))$.
Then up to $\tH$-conjugacy, we may assume that $\iota=\prod_{i=1}^{u}\iota_i^{n_i}$ where $\iota_i\in\Lin(H_0/M\C_{H_0}(N))$ and $n=\sum_{i=1}^un_i$ satisfy that $\iota_i$ and $\iota_j$ are not $H$-conjugate if $i\ne j$.
From this, $\tH_\iota=\prod_{i=1}^u H_{\iota_i}\wr\fS_{n_i}$.

In (a) we may assume further that $(G,N,\theta)\geqslant_{(g),*}(H,M,\vhi)$ is given by $(\cP,\cP')$ normally with respect to $H_0$.
Let $H_{i}=H_{\iota_i}$.
By the assumption, there exists a projective $\CC$-representation $\cP_{i}''$ associated with $(H_{i},M,\vhi)$ such that $(\cP,\cP''_{i})$ gives \[(G,N,\theta)\geqslant_{(g),*}(H_{i},M,\vhi),\] and $\cP''_{i}(h)=\iota_i(h)\cP'(h)$ for any $h\in H_0$.
Thus by the proofs of (a) and Theorem~\ref{thm:direct-prod}, we can construct a projective $\CC$-representation $\widetilde\cP''$ associated with $(\tH_\iota,\tM,\tvhi)$ such that $(\widetilde\cP,\widetilde\cP'')$ gives
\[(\tG\wr\fS_{n},\tN,\ttheta)\geqslant_{(g),*}(\tH_\iota,\tM,\tvhi).\]
In fact, if we let $\mathcal R'_i\colon\fS_{n_i}\to\GL_{n_i\vhi(1)}(\CC)$ be the linear $\CC$-representations of the symmetric group $\fS_{n_i}$ defined as in \cite[Definition~2.20]{Sp18} and let \[\widetilde\cP_{i}''((h_1,\ldots,h_{n_i})\sigma):=(\cP_{i}''(h_1)\otimes\cdots\otimes\cP_{i}''(h_{n_i}))\mathcal R'_i(\sigma)\]
for every $h_1,\ldots,h_{n_i}\in H_{\iota_i}$ and $\sigma\in\fS_{n_i}$, then \[\widetilde\cP''(h_1',\ldots,h_u')=\widetilde\cP_{1}''(h_1')\otimes\cdots\otimes\widetilde\cP_{u}''(h_u')\]
for $h_i'\in H_i\wr\fS_{n_i}$ with $i=1,\ldots,u$.
So $\cP''(h)=\iota(h)\widetilde\cP'(h)$ for any $h\in \tH_0$ and we complete the proof.
\end{proof}

In the following statement we consider character triples where the characters but not the groups coincide.

\begin{thm}\label{thm:butt-thm}
\begin{enumerate}[\rm(a)]
\item Suppose that $(G_1,N,\theta)\geqslant_{(g),*}(H_1,M,\vhi)$ with $*\in\{c,b\}$ and $N\unlhd G_2$. 
Assume that via the canonical morphism $\eps_i\colon G_i\to\Aut(N)$ (with kernel $\C_{G_i}(N)$) one has $\eps_1(G_1)=\eps_2(G_2)$.
Let $H_2=\eps_2^{-1}\eps_1(H_1)$.
	Then 
	\[(G_2,N,\theta)\geqslant_{(g),*}(H_2,M,\vhi).\]
\item In (a), if we assume further that $(G_1,N,\theta)\geqslant_{(g),*}(H_1,M,\vhi)$ is normal with respect to $H_{1,0}$, then $(G_2,N,\theta)\geqslant_{(g),*}(H_2,M,\vhi)$ is normal with respect to $H_{2,0}$, where $H_{2,0}$ is any subgroup of $H_2$ such that $M\le H_{2,0}\unlhd H_2$ and $\eps_1(H_{1,0})=\eps_2(H_{2,0})$.
\end{enumerate}
\end{thm}

\begin{proof}
(a)	 Note that $\eps_1(NH_1)=\eps_2(NH_2)$.
Let $\cT_1\subseteq H_1$ be a complete set of representatives of the $M\C_{G_1}(N)$-cosets in $H_1$ with $1\in\cT_1$.
Then $\cT_1$ can be extended to a complete set $\cT_1'$ of representatives of the $N\C_{G_1}(N)$-cosets in $G_1$.
Now for every $t\in\cT_1'$ we choose $\hat t\in G_2$ with $\eps_1(t)=\eps_2(\hat t)$, with $\hat 1=1$.
Then $\cT_2'=\{\hat t\mid t\in\cT'_1\}$ represents all $N\C_{G_2}(N)$-cosets in $G_2$, and \[\cT_2=\{\hat t\mid t\in\cT_1\}=H_2\cap \cT_2'\] represents all $M\C_{G_2}(N)$-cosets in $H_2$.

Let $\Irr(\Z(N)\mid\theta)=\{\mu\}$ and let $\widehat\mu\colon \C_{G_2}(N)\to\CC^\ti$ be a map such that $\widehat\mu(cz)=\widehat\mu(c)\mu(z)$ for every $c\in\C_{G_2}(N)$ and $z\in\Z(N)$.
For example, if we write $\C_{G_2}(N)=\coprod_{i=1}^{t}c_i\Z(N)$ with $c_1=1$, then we can define $\widehat \mu(c_iz)=\mu(z)$ for $z\in\Z(N)$ (see e.g. \cite[p.~83]{Na18}).

Suppose that $(G_1,N,\theta)\geqslant_{(g),b}(H_1,M,\vhi)$ is given by $(\cP_1,\cP'_1)$.
Define maps \[\cP_2\colon G_2\to\GL_{\theta(1)}(\CC) \ \ \textrm{and}\ \ \cP_2'\colon H_2\to\GL_{\vhi(1)}(\CC) \]
by \[\cP_2(\hat t'nc)=\cP_1(\hat t')\cP_1(n)\widehat\mu(c)\ \ \textrm{and}\ \ \cP_2'(\hat tmc)=\cP_1'(\hat t)\cP_1'(m)\widehat\mu(c) \]
for every $\hat t'\in\cT_2'$, $n\in N$, $c\in\C_{G_2}(N)$,  $\hat t\in\cT_2$ and $m\in M$.

As in the proof of \cite[Thm.~2.16]{Sp18}, we can show that $\cP_2$ and $\cP_2'$ are well-defined projective representations associated with $(G_2,N,\theta)$ and $(H_2,M,\vhi)$ respectively.
Similar to the proofs of \cite[Thm.~2.16]{Sp18} and \cite[Thm.~5.3]{Sp17}, one has that $(\cP_2\rceil_{NH_2},\cP_2')$ gives \[(NH_2,N,\theta)\geqslant_{*}(H_2,M,\vhi).\]
By Remark~\ref{rmk:gen-isomorphisms} and Corollary~\ref{rmk:gen-isomorphisms-b} this completes the proof.

(b) We assume further that \[\{\,t\in\cT_1\mid \eps_1(t)\in \eps_1(H_{1,0})\,\}\subseteq H_{1,0}\] and 
\[\{\,\hat t\in\cT_2\mid \eps_2(\hat t)\in \eps_2(H_{2,0})\,\}\subseteq H_{2,0}.\]
Let $\iota_2\in\Lin(H_{2,0}/M\C_{H_{2,0}}(N))$.
Since \[H_{2,0}/M\C_{H_{2,0}}(N)\cong H_{2,0}\C_{G_2}(N)/M\C_{G_2}(N)\cong \eps_2(H_{2,0})/\eps_2(M),\] $\iota$ can be regarded as an irreducible character of $\eps_2(H_{2,0})/\eps_2(M)$.
So $\iota_2$ corresponds naturally to a character $\iota_1\in\Lin(H_{1,0}/M\C_{H_{1,0}}(M))$ so that $\eps_1((H_1)_{\iota_1})=\eps_2((H_2)_{\iota_2})$ and $\C_{G_i}(N)\le (H_i)_{\iota_i}$.
In particular, $\iota_1(t)=\iota_2(\hat t)$ for $t\in\cT_1\cap H_{1,0}$.

In (a) we may assume further that $(G_1,N,\theta)\geqslant_{(g),*}(H_1,M,\vhi)$ is given by $(\cP_1,\cP_1')$ normally with respect to $H_{1,0}$. 
Then there exists a projective $\CC$-representation $\cP_1''$ associated with $((H_1)_{\iota_1},M,\vhi)$ such that $(\cP_1,\cP_1'')$ gives \[(G_1,N,\theta)\geqslant_{(g),*}((H_1)_{\iota_1},M,\vhi),\] and $\cP''_1(h)=\iota_1(h)\cP'_1(h)$ for any $h\in H_{1,0}$.
Define map \[\cP_2''\colon (H_2)_{\iota_2}\to\GL_{\vhi(1)}(\CC) \]
by \[\cP_2''(\hat tmc)=\cP_1''(\hat t)\cP_1''(m)\widehat\mu(c) \]
for every $\hat t\in\cT_2$ , $c\in\C_{G_2}(N)$ and $m\in M$.
Then $\cP''_2(h)=\iota_2(h)\cP'_2(h)$ for any $h\in H_{2,0}$.
Now by the proof of (a), $(\cP_2,\cP_2'')$ gives \[(G_2,N,\theta)\geqslant_{(g),*}((H_2)_{\iota_2},M,\vhi),\] and thus we complete the proof.
\end{proof}


\section{Weights and Clifford theory}\label{sec:conclusion}

A fundamental approach on how to get an equivariant bijection between the irreducible Brauer characters and weights for a finite group from such equivariant bijection of certain normal subgroups in terms of modular character triples was established in \cite{FLZ22,FLZ23,MRR23}.
In this section, we establish some statements for equivariant bijections between the irreducible characters and weights of a finite group and its normal subgroups in terms of character triples.

\subsection{Covering of weights}
Let $G$ be a finite group. 
A \emph{weight} of $G$ is a pair $(R,\vhi)$, where $R$ is a (possibly trivial)
$p$-subgroup of $G$ and $\vhi\in\dz(\N_G(R)/R)$.
Let $\Alp^0(G)$ denote the set of weights of $G$. 
The $G$-orbit of $(R,\vhi)$ is denoted by $\overline{(R,\vhi)}$ and we define
$\Alp(G)=\Alp^0(G)/\!\sim_G$. Sometimes we also write $\overline{(R,\vhi)}$
simply as $(R, \vhi)$ when no confusion can arise.
For $\nu\in\Lin_{p'}(\Z(G))$, we denote by $\Alp^0(G\mid \nu)$ the set of
weights $(R,\vhi)$ of $G$ with $\vhi\in\Irr(\N_G(R)\mid\nu)$ and write
\[\Alp(G\mid\nu)=\Alp^0(G\mid \nu)/\!\sim_G.\]
The group $\Lin_{p'}(G)$ acts on $\Alp^0(G)$ by $\mu.(R,\vhi)=(R,\mu'\vhi)$
where $\mu'$ is the restriction of $\mu\in\Lin_{p'}(G)$ to $\N_G(R)$
(sometimes we also write $\mu$ for $\mu'$ when no confusion can arise); see
\cite[Lemma~2.4]{BS22}.
This induces an action of $\Lin_{p'}(G)$ on $\Alp(G)$.

Each weight may be assigned to a unique block. Let $B$ be a block of $G$. 
A weight $(R,\vhi)$ of $G$ is said to be a $B$-weight if
$\bl_{\N_G(R)}(\vhi)^G=B$. We denote the set of $B$-weights by $\Alp^0(B)$.
Let $\Alp(B)=\Alp^0(B)/\!\sim_G$. If $b$ is a union of blocks of $G$, then we
define $\Alp(b)=\bigcup_{B\in b}\Alp(B)$.

In \cite{BS22}, Brough and Sp\"ath defined a relationship ``covering" for
weights between a finite group and its normal subgroups. To state this, we
first recall the Dade--Glauberman--Nagao correspondence from \cite{NS14}.
Let $N\unlhd M$ be finite groups such that $M/N$ is a $p$-group.
Suppose that $b$ is an $M$-invariant block of $N$ with defect group
$D_0\le \Z(M)_p\cap N$, $B$ is the unique block of $M$ covering $b$ and $D$
is a defect group of $B$. Let $B'$ be the Brauer correspondent of $B$.
Denote $L=\N_N(D)$ and let $b'$ be the unique block of $L$ covered by $B'$.
Then by \cite[Thm.~5.2]{NS14}, there is a natural bijection
$\pi_D\colon\Irr_D(b)\to \Irr_D(b')$, called the \emph{Dade--Glauberman--Nagao
	(DGN) correspondence}. Here, $\Irr_D(b)$ denotes the set of $D$-invariant
irreducible characters in $b$.

Suppose that $G\unlhd \tG$ are finite groups.
Let $(R,\varphi)$ be a weight of $G$ and set $N=\N_G(R)$. Let $M$ be a
subgroup of $\N_{\tG}(R)_\vhi$ such that $N\le M$ and $M/N$ is a $p$-group.
We fix a defect group $\tR/R$ of the unique block of $M/R$ which covers
$\bl_{N/R}(\vhi)$.
Then $R=\tR\cap N$, $M=N\tR$ and \[\N_{M/R}(\tR/R)=(\tR/R)\ti \C_{N/R}(\tR/R).\]
Note that $\pi_{\tR/R}(\varphi)\in\dz(\C_{N/R}(\tR/R))$. Denote by
$\overline\pi_{\tR/R}(\vhi)$ the associated character of
$\N_{M/R}(\tR/R)/(\tR/R)$ which lifts to
$\pi_{\tR/R}(\vhi)\times 1_{\tR/R}\in\Irr(\N_{M/R}(\tR/R))$.
Also note that \[\N_{M/R}(\tR/R)/(\tR/R)\cong \N_G(\tR)\tR/\tR\] is normal in
$\N_{\tG}(\tR)/\tR$.
Following Brough--Sp\"ath \cite{BS22}, a weight $(\tR,\tvhi)$ of $\tG$ with
\[\tvhi\in\dz(\N_{\tG}(\tR)/\tR\mid \overline\pi_{\tR/R}(\vhi))\] is said to
\emph{cover} $(R,\varphi)$.

If $(R,\vhi)$ is a weight of $G$ then we write $\Alp^0(\tG\mid(R,\vhi))$ for
the set of those $(\tR,\tvhi)\in\Alp^0(\tG)$ covering $(R,\vhi)$.
If $(\tR,\tvhi)$ is a weight of $\tG$, then we write $\Alp^0(G\mid(\tR,\tvhi))$
for the set of those $(R,\vhi)\in\Alp^0(G)$ covered by $(\tR,\tvhi)$.

For $(\tR,\tvhi)\in\Alp^0(\tG)$ and $(R,\vhi)\in\Alp^0(G)$, we say that
$\overline{(\tR,\tvhi)}$ \emph{covers} $\overline{(R,\vhi)}$ if $(\tR,\tvhi)$
covers $(R^g,\vhi^g)$ for some $g\in\tG$. If $(R,\vhi)$ is a weight of $G$ we
write $\Alp(\tG\mid\overline{(R,\vhi)})$ for the set of those
$\overline{(\tR,\tvhi)}\in\Alp(\tG)$ covering $\overline{(R,\vhi)}$.
If $(\tR,\tvhi)$ is a weight of $\tG$, then we write
$\Alp(G\mid\overline{(\tR,\tvhi)})$ for the set of those
$\overline{(R,\vhi)}\in\Alp(G)$ covered by $\overline{(\tR,\tvhi)}$.
For a subset $\cA\subseteq\Alp(G)$ we define
\[\Alp(\tG\mid\cA)
=\bigcup_{\overline{(R,\vhi)}\in\cA} \Alp(\tG\mid\overline{(R,\vhi)})\]
and for a subset $\widetilde\cA\subseteq\Alp(\tG)$, we define
\[\Alp(G\mid\widetilde\cA)=\bigcup_{\overline{(\tR,\tvhi)}\in\widetilde\cA}
\Alp(G\mid\overline{(\tR,\tvhi)}). \]

Let $G$ and $\tG$ be finite groups with $G\unlhd \tG$, and let $(R,\vhi)$ be a weight of $G$.
We assume further that $\tG/G$ is abelian and $\vhi$ extends to $\N_{\tG}(R)_\vhi$. 
By \cite[\S2.C]{BS22}, $\tR/R\cong\tR\N_G(R)/\N_G(R)$ is a Sylow $p$-subgroup of $\N_{\tG}(R)_\vhi/\N_G(R)$.
By \cite[Thm.~3.8]{Sp13} or \cite[Thm.~5.13]{NS14}, there exists an $\N_{\tG}(R)_\vhi$-invariant extension $\hvhi$ of $\vhi$ to $\tR\N_G(R)$ and a bijection
\[\Delta_\vhi\colon\rdz(\N_{\tG}(R)_\vhi\mid\hvhi)\to\dz(\N_{\tG}(\tR)_\vhi/\tR\mid\overline\pi_{\tR/R}(\vhi))\]
with $\bl_{\N_{\tG}(R)_\vhi}(\theta)=\bl_{\N_{\tG}(\tR)_\vhi}(\theta')^{\N_{\tG}(R)_\vhi}$ for every $\theta\in\rdz(\N_{\tG}(R)_\vhi\mid\hvhi)$.

\subsection{A criterion to establish $\geqslant_{(g),b}$}	
In \cite{BS22}, Brough and Sp\"ath presented a criterion for the inductive conditions of Alperin weight conjecture, which is mainly devoted to establishing block isomorphisms between modular character triples.
Now we give a similar result to establish $\geqslant_{(g),b}$ for character triples, which is useful for groups of Lie type.

\begin{thm}\label{thm:criterion-block}
Let $G$, $\tG'$ and $\tG$ be normal subgroups of a finite group $A$ such that for some subgroup $E$ of $A$ we have $A=\tG\rtimes E$.
Let $\tB$ be a union of blocks of $\tG$ which is $\Lin_{p'}(\tG/G)$-stable. Suppose that $\widetilde\cI\subseteq\Irr(\tB)\cap\Irr(\tG\mid 1_{\Z(\tG)_p})$ and $\widetilde\cA\subseteq \Alp(\tB)$ are $\Lin_{p'}(\tG/G)\rtimes E_{\tB}$-stable.
Let $\cI=\Irr(G\mid\widetilde\cI)$ and $\cA=\Alp(G\mid\widetilde\cA)$, and let $B$ be the union of blocks of $G$ covered by a block in $\tB$.
Suppose that the following hold.

	\begin{enumerate}[\rm(i)]
		\item \begin{enumerate}[\rm(a)]
		\item Both $\tG/G$ and $E$ are abelian,
		\item $\tG'/G=(\tG/G)_p$ and $\C_{A}(G)=\Z(\tG)$,
			\item for any character $\chi\in\cI$, one has that $\tG'\subseteq \tG_\chi$ and $\chi$ extends to its stabilizer in $\tG$, and
			\item for any element $\overline{(R,\vhi)}\in\cA$, the weight character $\vhi$ extends to its stabilizer in $\N_{\tG}(R)$.
			\end{enumerate}
			\item For every $\tchi\in\widetilde\cI$, there exists some character $\chi_0\in\Irr(G\mid\tchi)$ such that $(\tG\rtimes E)_{\chi_0}=\tG_{\chi_0}\rtimes E_{\chi_0}$ and
			$\chi_0$ extends to $G\rtimes E_{\chi_0}$.
			\item For every element $\overline{(\tR,\tvhi)}\in\widetilde\cA$, there exists some  $\overline{(R_0,\vhi_0)}\in\Alp(G\mid \overline{(\tR,\tvhi)})$ such that
			$(\tG\rtimes E)_{R_0,\vhi_0} = \tG_{R_0,\vhi_0} (G\rtimes E)_{R_0,\vhi_0}$ and	$\vhi_0$ extends to $(G\rtimes E)_{R_0,\vhi_0}$.
			\item 
		There exists an $\Lin_{p'}(\tG/G) \rtimes E_{\tB}$-equivariant bijection $\widetilde \varOmega\colon \widetilde\cI \to \widetilde\cA$ such that
		\begin{enumerate}[\rm(a)]
				\item $\widetilde\varOmega$ preserves blocks, and
			\item if the character $\tchi$ in (ii) and the weight $(\tR,\tvhi)$ in (iii) satisfying \[\overline{(\widetilde R,\widetilde\vhi)}=\widetilde\varOmega(\tchi),\]  then the character $\chi_0$ in (ii) and the weight $(R_0,\vhi_0)$ in (iii) can be chosen to satisfy $\bl(\hchi)=\bl(\hvhi)^{\tG_{\chi_0}}$, where $\hchi\in\Irr(\tG_{\chi_0}\mid\chi_0)$ is the  Clifford correspondent of $\tchi$ and $\hvhi$ is the extension of $\vhi_0$ to $\N_{\tG}(R_0)_{\vhi_0}/R_0$ such that via $\Delta_{\vhi_0}$ and the induction it corresponds to $(\tR,\tvhi)$.
		\end{enumerate}		
	\end{enumerate}
	Then there exists a blockwise $A_B$-equivariant bijection $\varOmega\colon\cI\to\cA/\!\sim_{\tG'}$ such that for every $\chi\in\cI$, there is a weight $(R,\vhi)$ of $G$ whose $\tG'$-orbit corresponds to of $\chi$ via $\varOmega$ satisfying that $\N_A(R)_\vhi\subseteq A_\chi$ and 
	\[ (A_\chi,G,\chi) \geqslant_{(g),b} (\N_A(R)_\vhi,\N_G(R),\vhi)\]
is normal with respect to $\N_{\tG'}(R)_\vhi$.
\end{thm}

\begin{proof}
Let $\widetilde\cI_T$ be an $(\Lin_{p'}(\tG/G) \rtimes E_{\tB})$-transversal in $\widetilde\cI$.
For every $\tchi\in\widetilde\cI_T$, we fix a character $\chi_0\in\cI$ with the property (ii) and let $\cI_0\subseteq \cI$ be the subset formed by those $\chi_0$.
We claim that $\cI_0$ is a $A_B$-transversal in $\cI$.
In fact, for every character $\tau\in\cI$ we let $\widetilde\tau\in\Irr(\tG\mid\tau)$.
Then there exists a unique character $\tchi\in\widetilde\cI_T$ with $\widetilde\tau=\tchi^e\eta$ for $e\in E_B$ and $\eta\in\Lin_{p'}(\tG/G)$.
From this $\chi_0$ and $\tau^{e^{-1}}$ are contained in $\Irr(G\mid\tchi)$ which implies that they are $\tG$-conjugate.
So there exists a unique element in $\cI_0$ to which $\tau$ is $A_B$-conjugate and this proves the claim.

Let $\widetilde\cA_T=\widetilde\varOmega(\widetilde\cI_T)$. By the equivariant property of $\widetilde\varOmega$, we know that $\widetilde\cA_T$ is an $\Lin_{p'}(\tG/G) \rtimes E_{\tB}$-transversal in $\widetilde\cA$.
As above for every element $\overline{(\tR,\tvhi)}\in\widetilde\cA_T$ we associate a weight $(R_0,\vhi_0)$ of $G$ such that $\overline{(\tR,\tvhi)}$ covers $\overline{(R_0,\vhi_0)}$ and the property (iii) holds.
Let those $\overline{(R_0,\vhi_0)}$ form the set $\cA_0$, an $A_B$-transversal in $\cA$.

For $\chi_0\in\cI_0$, $\tchi\in\Irr(\tG\mid\chi_0)\cap\widetilde\cI_T$ and $\widetilde\varOmega(\tchi)=\overline{(\tR,\tvhi)}$, we define $\varOmega_0(\chi_0)=\overline{(R_0,\vhi_0)}$, where $\overline{(R_0,\vhi_0)}$ is the unique element of $\cA_0\cap\Alp(G\mid\overline{(\tR,\tvhi)})$.
Now we prove that 
\begin{equation}\label{equ-stablizer=}
	(\tG\rtimes E)_{\chi_0}=(\tG\rtimes E)_{R_0,\vhi_0}\tG'.
\addtocounter{thm}{1}\tag{\thethm}
\end{equation}

First by the proof of \cite[Thm.~2.12]{Sp12}, one has 
\begin{equation}\label{equ-global-sta1}
\tG_{\chi_0}=\bigcap_{\{\ze\in\Lin(\tG/G)\mid \tchi\ze=\tchi\}}\ker(\ze). 
\addtocounter{thm}{1}\tag{\thethm}
\end{equation}
In addition, $(\tG\rtimes E)_{\chi_0}=\tG_{\chi_0}\rtimes E_{\chi_0}$ implies
\begin{equation}\label{equ-global-sta2}
E_{\chi_0}=\{\, e\in E\mid \tchi^e=\tchi\ze\ \textrm{for some}\ \ze\in\Lin(\tG/G) \,\}. 
\addtocounter{thm}{1}\tag{\thethm}
\end{equation}
Since $\tG'\subseteq \tG_{\chi_0}$, one has that in (\ref{equ-global-sta1}) (or (\ref{equ-global-sta2})) $\tchi\ze=\tchi$ (or $\tchi^e=\tchi\ze$) only if $\ze\in\Lin_{p'}(\tG/G)$.
Let $\tG''/G$ be the Hall $p'$-subgroup of $\tG/G$.
Then \[\tG_{\chi_0}=\tG'\bigg(\Big(\bigcap_{\{\ze\in\Lin_{p'}(\tG/G)\mid \tchi\ze=\tchi\}}\ker(\ze)\Big)\cap\tG''\bigg)\]

On the other hand, by \cite[Lemma~2.15]{BS22},
\[\tG_{\overline{(R_0,\vhi_0)}}=\tG'_{\overline{(R_0,\vhi_0)}}\bigg(\Big(\bigcap_{\{\ze\in\Lin_{p'}(\tG/G)\mid \tvhi\ze=\tvhi\}}\ker(\ze)\Big)\cap\tG''\bigg).\]
Note that $\widetilde\varOmega$ is $\Lin_{p'}(\tG/G) \rtimes E_{\tB}$-equivariant, so $\tG_{\chi_0}=\tG_{\overline{(R_0,\vhi_0)}}\tG'$. 
Since $(\tG\rtimes E)_{R_0,\vhi_0} = \tG_{R_0,\vhi_0} (G\rtimes E)_{R_0,\vhi_0}$, by the proof of \cite[Thm.~3.3]{BS22}, one has
\[ E_{\overline{(R_0,\vhi_0)}}=\{\, e\in E\mid \overline{(\tR,\tvhi)}^e=\overline{(\tR,\tvhi\ze)}\ \textrm{for some}\ \ze\in\Lin_{p'}(\tG/G) \,\}\]
which implies $E_{\chi_0}=E_{\overline{(R_0,\vhi_0)}}$.

Therefore, we complete the proof of (\ref{equ-stablizer=}) and from this we can define a bijection \[\varOmega\colon\cI\to\cA/\!\sim_{\tG'}\] by letting $\varOmega(\chi_0^x)$ to be the $\tG'$-orbit of $\varOmega_0(\chi_0)^x$ for every $\chi_0\in\cI_0$ and $x\in A_B$.
By definition, $\varOmega$ is an $A_B$-equivariant bijection, and by (iv.b) it can be chosen to satisfies $\bl(\hchi)=\bl(\hvhi)^{\tG_{\chi_0}}$, where $\hchi\in\Irr(\tG_{\chi_0}\mid\chi_0)$ is the  Clifford correspondent of $\tchi$ and $\hvhi$ is the extension of $\vhi_0$ to $\N_{\tG}(R_0)_{\vhi_0}/R_0$ such that via $\Delta_{\vhi_0}$ and the induction it corresponds to $(\tR,\tvhi)$.
So by Lemma~\ref{lemma:blockinductions1}, $\bl(\chi_0)=\bl(\vhi_0)^G$ and thus $\varOmega$ preserves blocks.
Then we shall have established this theorem if  we prove Lemma~\ref{lem:establish-iso}.
\end{proof}

\begin{lem}\label{lem:establish-iso}
$(A_{\chi_0},G,\chi_0) \geqslant_{(g),b} (\N_A(R_0)_{\vhi_0},\N_G(R_0),\vhi_0)$ is normal with respect to $\N_{\tG'}(R_0)_{\vhi_0}$ for every $\chi_0\in\cI_0$ and $\varOmega_0(\chi_0)=\overline{(R_0,\vhi_0)}$.
\end{lem}	

\begin{proof}
First, we have $\Irr(\Z(\tG)\mid\tchi)=\Irr(\Z(\tG)\mid\tvhi)$ which implies that \[\Irr(\C_A(G)\mid\hchi)=\Irr(\C_A(G)\mid\tvhi)=\{\eps\}\] for some $\eps\in\Irr(\Z(\tG))$.

Let $\cD_1$ be a linear $\CC$-representation of $\tG_{\chi_0}$ affording $\hchi$ and $\cD_2$ be a  linear $\CC$-representation of $G\rtimes E_{\chi_0}$ with $\cD_1\rceil_G=\cD_2\rceil_G$.
According to \cite[Lemma~2.11]{Sp12}, the map $\cP\colon A_{\chi_0}\to\GL_{\chi_0(1)}(\CC)$, defined by \[\cP(x_1x_2)=\cD_1(x_1)\cD_2(x_2)\ \ \textrm{for}\ \ x_1\in\tG_{\chi_0}\ \ \textrm{and} \ \ x_2\in G\rtimes E_{\chi_0},\] is a projective $\CC$-representation associated with the character triple $(A_{\chi_0},G,\chi_0)$.
Then \[\cP(c)=\eps(c)\ \  \textrm{for}\ \ c\in\C_A(G),\] and the factor set $\al$ of $\cP$ satisfies that \[\al(x_1x_2,x_1'x_2')=\mu_{x_2}(x_1')I_{\chi_0(1)}\ \ \textrm{for}\ \ x_1,x_1'\in\tG_{\chi_0}\ \ \textrm{and}\ \ x_2,x_2'\in G\rtimes E_{\chi_0},\] where $\mu_{x_2}\in\Lin_{p'}(\tG_{\chi_0}/G)$ is defined by $\hchi=\mu_{x_2}\hchi^{x_2}$.
Here $I_{\chi_0(1)}$ denotes the identity matrix of degree $\chi_0(1)$.

Let $\cD_1'$ be a linear $\CC$-representation of $\N_{\tG}(R_0)_{\vhi_0}$ affording $\hvhi$ and $\cD_2'$ be a  linear $\CC$-representation of $(G\rtimes E)_{R_0,\vhi_0}$ with \[\cD'_1\rceil_{\N_G(R_0)}=\cD'_2\rceil_{\N_G(R_0)}.\]
Also by \cite[Lemma~2.11]{Sp12}, the map $\cP'\colon A_{R_0,\vhi_0}\to\GL_{\vhi_0(1)}(\CC)$, defined by \[\cP'(x_1x_2)=\cD_1'(x_1)\cD_2'(x_2)\ \ \textrm{for}\ \ x_1\in\N_{\tG}(R_0)_{\vhi_0}\ \ \textrm{and}\ \ x_2\in (G\rtimes E)_{R_0,\vhi_0},\] is a projective $\CC$-representation associated with the character triple $(\N_A(R_0)_{\vhi_0},\N_G(R_0),\vhi_0)$.
Then $\cP'(c)=\eps(c)I_{\vhi_0(1)}$ for $c\in\C_A(G)$, and the factor set $\al'$ of $\cP'$ satisfies that \[\al'(x_1x_2,x_1'x_2')=\mu_{x_2}'(x_1')\ \ \textrm{for}\ \ x_1,x_1'\in\N_{\tG}(R_0)_{\vhi_0}\ \ \textrm{and} \ \ x_2,x_2'\in (G\rtimes E)_{R_0,\vhi_0},\] where $\mu_{x_2}'\in\Lin_{p'}(\N_{\tG}(R_0)_{\vhi_0}/\N_G(R_0))$ is defined by $\hvhi=\mu_{x_2}'\hvhi^{x_2}$.

Recall that $\Delta_{\vhi_0}$ is $\Lin_{p'}(\N_{\tG}(\tR)_{\vhi_0}/\N_G(\tR))\rtimes \N_A(R_0)_{\vhi_0}$-equivariant.
Since $\widetilde\varOmega$ is $\Lin_{p'}(\tG/G)\rtimes E_B$-equivariant, we can deduce that $\mu_{x_2}\rceil_{\N_A(R_0)_{\vhi_0}}=\mu_{x_2}'$.
So $(\cP,\cP')$ gives \[(A_{\chi_0},G,\chi_0) \geqslant_{(g),c} (\N_A(R_0)_{\vhi_0},\N_G(R_0),\vhi_0).\]
Note that $\bl(\hchi)=\bl(\hvhi)^{\tG_{\chi_0}}$.
By Lemma~\ref{lemma:blockinductions1}, we have $\bl(\Res^{\tG_{\chi_0}}_J(\hchi))=\bl(\Res^{\N_{\tG}(R_0)_{\vhi_0}}_{\N_J(R_0)_{\vhi_0}}(\hvhi))^J$ for every $G\le J\le \tG_{\chi_0}$.

As in the proof of Theorem~\ref{thm:criterion-geqslant_{(g),b}}, the factor set $\al$ determines a group $\widehat A$.
Let $Z$ be a finite subgroup of $\CC^\ti$ that contains all the values of the factor set $\al$, and let $\widehat A$ be the group whose elements are $(a,z)$ where $a\in A_{\chi_0}$ and $z\in Z$, and whose multiplication is defined by \[(a_1,z_1)(a_2,z_2)=(a_1a_2,\al(a_1,a_2)z_1z_2)\ \ \textrm{for}\ \ a_1,a_2\in A_{\chi_0}\ \ \textrm{and}\ \ z_1,z_2\in Z.\]
The map $\eps\colon \widehat A\to A_{\chi_0}$ given by $(a,z)\to a$ is an epimorphism.
Then $\cP$ lifts to a linear representation $\widehat\cD$ of $\widehat A$, that is defined by $\widehat\cD(a,z)=z\cP(a)$ for $a\in A_{\chi_0}$ and $z\in Z$. 
Let $G_1=G\ti 1$, $\chi_1=\chi_0\circ\eps\rceil_{G_1}$ and $\tG_2=\eps^{-1}(\tG_{\chi_0})$.

Let $\widehat N=\eps^{-1}(\N_A(R_0)_{\vhi_0})$, $N_1=G_1\cap\eps^{-1}(\N_G(R_0))$ and  $\vhi_1=\vhi_0\circ\eps\rceil_{N_1}$.
Analogously, $\cP'$ lifts to a linear representation $\widehat\cD'$ of $\N_A(R_0)_{\vhi_0}$, that is defined by $\widehat\cD'(a,z)=z\cP'(a)$ for $a\in \N_A(R_0)_{\vhi_0}$ and $z\in Z$. 
As in the proof of Theorem~\ref{thm:criterion-geqslant_{(g),b}}, $(\widehat\cD,\widehat\cD')$ gives $(\widehat A, G_1,\chi_1)\geqslant_{(g),c}(\widehat N, N_1,\vhi_1)$.

Let $\tchi_1$ and $\tvhi_1$ be the characters of $\widehat A$ and $\widehat N$ afforded by $\widehat\cD$ and $\widehat\cD'$ respectively.
Similar as the proof of \cite[Prop.~4.2]{CS15}, we conclude that \[\bl(\Res^{\widehat A}_J(\tchi_1))=\bl(\Res^{\widehat N}_{J\cap\widehat N}(\tvhi_1))^J\ \ \textrm{for every}\ \ G_1\le J\le \tG_2.\]
Then by Lemma~\ref{lem:induced-blocks2}, there exists $\xi\in\Lin(\widehat A/\tG_2)$ such that $\bl(\Res^{\widehat A}_{J}(\xi\tchi_1))=\bl(\Res^{\widehat N}_{J\cap \widehat N}(\tvhi_1))^J$ for every subgroup $G_1\le J\le \widehat A$.
Thus by Lemma~\ref{lem:ectend-block-iso}, $(\xi\widehat\cD,\widehat\cD')$ gives \[(\widehat A, G_1,\chi_1)\geqslant_{(g),b}(\widehat N, N_1,\vhi_1).\]
We also regard $\xi$ as a linear character of $E_{\chi_0}\cong G\rtimes E_{\chi_0}/G$.
In the construction of $\cP$, we may change $\cD_2$ by $\xi\cD_2$ so that the factor set of $\cP$ remains unchanged. 
Accordingly $(\cP,\cP')$ gives \[(A_{\chi_0},G,\chi_0) \geqslant_{(g),b} (\N_A(R_0)_{\vhi_0},\N_G(R_0),\vhi_0).\]

Finally, we show that $(\cP,\cP')$ gives \[(A_{\chi_0},G,\chi_0) \geqslant_{(g),b} (\N_A(R_0)_{\vhi_0},\N_G(R_0),\vhi_0)\] normally with respect to $\N_{\tG'}(R_0)_{\vhi_0}$.
Let $\iota\in\Irr(\N_{\tG'}(R_0)_{\vhi_0}/\N_G(R_0)\Z(\tG'))$. Then $\iota$ is $\N_G(R_0)_{\vhi_0}$-stable, and thus \[\N_A(R_0)_{\vhi_0,\iota}=\N_G(R_0)_{\vhi_0}(G\rtimes E)_{R_0,\vhi_0,\iota}.\]
Let $\tG''/G$ be the Hall $p'$-subgroup of $\tG/G$ and let $\widehat\iota\in\Irr(\N_G(R_0)_{\vhi_0}\mid 1_{\N_{\tG''}(R_0)_{\vhi_0}})$ be the (unique) extension of $\iota$ to $\N_G(R_0)_{\vhi_0}$. 
We have $\N_A(R_0)_{\vhi_0,\iota}=\N_A(R_0)_{\vhi_0,\widehat\iota}$.
Then the map \[\cP''\colon A_{R_0,\vhi_0,\iota}\to\GL_{\vhi_0(1)}(\CC),\] defined by \[\cP''(x_1x_2)=\widehat\iota(x_1)\cD_1'(x_1)\cD_2'(x_2) \ \ \textrm{for} \ \ x_1\in\N_{\tG}(R_0)_{\vhi_0}\ \ \textrm{and} \ \ x_2\in (G\rtimes E)_{R_0,\vhi_0,\iota},\] is a projective $\CC$-representation associated with the character triple $(\N_A(R_0)_{\vhi_0,\iota},\N_G(R_0),\vhi_0)$.
Then $\cP''(x)=\iota(x)\cP'(x)$ for $x\in\N_{\tG'}(R_0)_{\vhi_0}$ and as above we can prove that $(\cP,\cP'')$ gives \[(A_{\chi_0},G,\chi_0) \geqslant_{(g),b}(\N_A(R_0)_{\vhi_0,\iota},\N_G(R_0),\vhi_0).\]
Thus we complete the proof.
\end{proof}

We end this subsection by proving the following result, which is devoted to establishing condition (iv.b) of Theorem~\ref{thm:criterion-block} for some special situations.

\begin{prop}\label{prop:condition-iv-b}
In Theorem~\ref{thm:criterion-block}, we assume further that $\tG/G$ is cyclic and for every block $\widetilde b$ in $\tB$ and every $\tchi\in\Irr(\widetilde b)\cap\widetilde\cI$, the number of blocks of $G$ covered by $\widetilde b$ equals the cardinality of $\Irr(G\mid\tchi)$.
If the conditions (i), (ii), (iii) and (iv.a) are satisfied, then the condition (iv.b) holds.
\end{prop}

\begin{proof}
Let $\tchi$ and $\chi_0$ be characters in (ii) and let $(\tR,\tvhi)$ and $(R_0,\vhi_0)$ be the weight in (iii) such that $\overline{(\widetilde R,\widetilde\vhi)}=\widetilde\varOmega(\tchi)$.

We first prove that the character $\chi_0$ and the weight $(R_0,\vhi_0)$ can be chosen in the same block of~$G$.
Let $b=\bl(\chi_0)$ and suppose that $(R_0,\vhi_0)$ is a $b^g$-weight for some $g\in\tG$.
By our assumption, $\tG_b=\tG_{\chi_0}$.
We may assume that $g\notin\tG_{b}$.
By the proof of Theorem~\ref{thm:criterion-block}, we have \[\tG_{\chi_0}\rtimes E_{\chi_0}=(\tG\rtimes E)_{\chi_0}=\tG'(\tG\rtimes E)_{\overline{(R_0,\vhi_0)}}.\]
Since $\overline{(R_0,\vhi_0)}\in\Alp(b^g)$, we have \[\tG_{\chi_0}\rtimes E_{\chi_0}=\tG'(\tG\rtimes E)_{\overline{(R_0,\vhi_0)}}\le\tG'(\tG\rtimes E)_{b^g}=(\tG\rtimes E)_{b^g}.\]
Also $(\tG\rtimes E)_{\chi_0^g}=\tG_{\chi_0}E^{g}_{\chi_0}\le (\tG\rtimes E)_{b^g}$.
In particular, both $E_{\chi_0}$ and $E_{\chi_0}^g$ are contained in $(\tG\rtimes E)_{b^g}$.
For $e\in E_{\chi_0}$, we have $[e,g]=e^{-1}e^{g}\in (\tG\rtimes E)_{b^g}$. This  implies that \[[\langle \tG_{\chi_0}, g\rangle,E_{\chi_0}]\subseteq \tG_{b^g}=\tG_b=\tG_{\chi_0}.\]
So $(\tG\rtimes E)_{\chi_0^g}=\tG_{\chi_0}E_{\chi_0}^g\le\tG_{\chi_0}E_{\chi_0}$ and from this $(\tG\rtimes E)_{\chi_0^g}=\tG_{\chi_0^g}\rtimes E_{\chi_0^g}$.
Thus the character $\chi_0^g$ also satisfies (ii) and $\chi_0^g$ and $(R_0,\vhi_0)$ lie in the same block.

Let $\hchi\in\Irr(\tG_{\chi_0}\mid\chi_0)$ is the  Clifford correspondent of $\tchi$.
By a theorem of Fong--Reynolds \cite[Thm.~5.5.10]{NT89}, $\bl(\hchi)$ is the unique block of $\tG_{\chi_0}=\tG_{\bl(\chi_0)}$ which covers $\bl(\chi_0)$ and is covered by $\bl(\tchi)$.
Recall that $\Delta_{\vhi_0}$ is a bijection between $\rdz(\N_{\tG}(R_0)_{\vhi_0}\mid\hvhi_0)$ and $\dz(\N_{\tG}(\tR)_{\vhi_0}/\tR\mid\overline\pi_{\tR/R}(\vhi_0))$ where $\hvhi_0$ is an extension of $\vhi_0$ to $\N_G(R_0)\tR$.
Let $\hvhi\in\rdz(\N_{\tG}(R_0)_{\vhi_0}\mid\tvhi_0)$ and $\hvhi'=\Delta_{\vhi_0}'(\hvhi)$ such that $\tvhi=\Ind^{\N_{\tG}(\tR)}_{\N_{\tG}(\tR)_{\vhi_0}}(\hvhi')$.
Then $\bl(\hvhi)=\bl(\hvhi')^{\N_{\tG}(R_0)_{\vhi_0}}$.
By \cite[Lemma~2.3]{KS15}, $\bl(\hvhi)^{\tG_{\chi_0}}$ covers $\bl(\vhi_0)^G=\bl(\chi_0)$, and $\bl(\tvhi)^{\tG}=\bl(\tchi)$ covers $\bl(\hvhi)^{\tG_{\chi_0}}$.
So $\bl(\hchi)=\bl(\hvhi)^{\tG_{\chi_0}}$ and this completes the proof.
\end{proof}

\subsection{A going-up property}

\begin{prop}\label{prop-iso-special-ell'}
	Let $G$ and $\tG$ be normal subgroups of a finite group $A$ such that $\tG/G$ is an abelian $p'$-group.
	Let $B$ be a union of blocks of $G$ which is  $A$-stable.
	Suppose that $\cI\subseteq \Irr(B)$ and $\cA\subseteq \Alp(B)$ are $A$-stable such that there is a blockwise $A$-equivariant bijection $\varOmega\colon\cI \ \to\ \cA$ such that for every $\chi\in\cI$ and  $\overline{(R,\vhi)}=\varOmega(\chi)$, one has
	\[(A_\chi,G,\chi) \geqslant_{b} (\N_A(R)_\vhi,\N_G(R),\vhi).\]
	Assume further that every  character $\chi\in\cI$ extends to $\tG_\chi$. 
	Then there exists a blockwise $(\Lin(\tG/G)\rtimes A)$-equivariant bijection \[\widetilde\varOmega\colon\Irr(\tG\mid\cI)\to\Alp(\tG\mid\cA)\] such that for every $\tchi\in\Irr(\tG\mid\cI)$ and $\overline{(\tR,\tvhi)}=\widetilde\varOmega(\tchi)$,  \[(A_{\tchi},\tG,\tchi)\geqslant_b(\N_A(\tR)_{\tvhi},\N_{\tG}(\tR),\tvhi).\]
\end{prop}

\begin{proof}
	Let $\chi\in\cI$ and $\overline{(R,\vhi)}=\varOmega(\chi)$ so that \[(A_\chi,G,\chi) \geqslant_{b} (\N_A(R)_\vhi,\N_G(R),\vhi),\] which is given by $(\cP,\cP')$.
	Let $\al$ and $\al'$ be the factor sets of $\cP$ and $\cP'$ respectively and let $\sigma^{(\chi)}$ be the isomorphism between character triples $(A_\chi,G,\chi)$ and $(\N_A(R)_\vhi,\N_G(R),\vhi)$ defined by $(\cP,\cP')$ as in Theorem~\ref{thm:iso-char-triple}.
By an analogous argument as in the proof of \cite[Prop.~4.7]{NS14} we can choose $\sigma^{(\chi)}$ such that \[\sigma_J^{(\chi)}(\xi)^x=\sigma_{J^x}^{(\chi^x)}(\xi^x)\ \ \textrm{for every}\ \ \chi\in\cI,\ G\le J\le A_\chi,\ \xi\in\Irr(J\mid\chi) \ \ \textrm{and}\ \  x\in\N_J(R).\]
Since $\tG/G$ is an abelian $p'$-group, a weight of $\tG$ covering $(R,\vhi)$ is of form $(R,\tvhi)$ with $\tvhi\in\Irr(\N_{\tG}(R)\mid\tvhi)$.

Let $\hchi\in\Irr(\tG_\chi\mid\chi)$ and $\hvhi=\sigma^{(\chi)}_{\tG_\chi}(\hchi)$. Then $\hchi$ is an extension of $\chi$ and $\hvhi$ is an extension of $\vhi$.
Let $\widehat A=\N_A(\tG_\chi)_{\hchi}$ and $\widehat N=\N_{\N_A(R)}(\tG_\chi)_{\hvhi}$.
Suppose that $\hchi$ is afforded by $\cQ\otimes\cP\rceil_{\tG_\chi}$ where $\cQ$ is a projective $\CC$-representation of $\tG_\chi/G$.
Then $\hvhi$ is afforded by $\cQ\rceil_{\N_{\tG}(R)_\vhi}\otimes\cP'\rceil_{\N_{\tG}(R)_\vhi}$.
Let $\widehat\cP$ be a projective $\CC$-representation associated with $(\widehat A,\tG_\chi,\hchi)$ such that $\widehat\cP\rceil_{\tG_\chi}=\cQ\otimes\cP\rceil_{\tG_\chi}$.
Suppose that $\beta$ is the factor set of $\widehat\cP$.
As in the proof of \cite[Prop.~3.9]{NS14}, we may write \[\widehat\cP(g)=\cU(g)\otimes\cP(g)\ \ \textrm{for}\ \ g\in\widehat A\] where $\cU$ is a projective $\CC$-representation of $\widehat A$ with factor set $\beta(\al\rceil_{\widehat A\ti\widehat A})^{-1}$.
Define \[\widehat\cP'=\cU\rceil_{\widehat N}\otimes\cP'\rceil_{\widehat N},\] which is a projective $\CC$-representation of associated with $(\widehat N,\N_{\tG}(R)_\vhi,\hvhi)$.
As in the proof of \cite[Prop.~3.9]{NS14} (using the proof of \cite[Thm.~(8.16)]{Na98}), we can prove that $(\widehat\cP,\widehat\cP')$ gives \[(\widehat A,\tG_\chi,\hchi)\geqslant_b(\widehat N,\N_{\tG}(R)_\vhi,\hvhi).\]

For $\tchi\in\Irr(\tG\mid\cI)$, we let $\chi\in\Irr(G\mid\tchi)$ and $\hchi\in\Irr(\tG_\chi\mid\chi)$ such that $\tchi=\Ind_{\tG_\chi}^{\tG}(\hchi)$.
Then we claim that $\tvhi:=\Ind^{\N_{\tG}(R)}_{\N_{\tG}(R)_\vhi}(\sigma^{(\chi)}_{\tG_\chi}(\hchi))$ is well-defined and independent of the choice of $\chi$.
In fact, if we take $\chi'\in\Irr(G\mid\tchi)$ and $\hchi'\in\Irr(\tG_{\chi'}\mid\chi')$ such that $\tchi=\Ind_{\tG_{\chi'}}^{\tG}(\hchi')$, then there exists $g\in\tG$ such that $\chi'=\chi^g$ and $\hchi'=\hchi^g$.
Since $\varOmega$ is $\tG$-equivariant, we may assume that $g\in\N_{\tG}(Q)$.
From this, \[\sigma_{\tG_{\chi}}^{(\chi)}(\hchi)^g=\sigma_{\tG_{\chi^g}}^{(\chi^g)}(\hchi^g)=\sigma_{\tG_{\chi'}}^{(\chi')}(\hchi')\] and hence $\Ind^{\N_{\tG}(R)}_{\N_{\tG}(R)_\vhi}(\sigma^{(\chi')}_{\tG_{\chi'}}(\hchi'))=\Ind^{\N_{\tG}(R)}_{\N_{\tG}(R)_\vhi}(\sigma^{(\chi)}_{\tG_\chi}(\hchi))$.
Thus the claim holds.

Let $\mathcal T$ be a complete set of representatives of $\tG$-orbit in $\cI$.
Then by the above argument, for every $\chi\in\mathcal T$ and $\varOmega(\chi)=\overline{(R,\vhi)}$, we can define a bijection 
\[\Irr(\tG\mid\chi)\to\Irr(\N_{\tG}(R)\mid\vhi),\quad\tchi\mapsto\tvhi.\]
Since $\tchi=\Ind_{\tG_\chi}^{\tG}(\hchi)$ and $\tvhi=\Ind^{\N_{\tG}(R)}_{\N_{\tG}(R)_\vhi}(\hvhi)$, by Clifford theory, we have $A_{\tchi}=\tG\widehat A$ and $\N_{A}(R)_{\tvhi}=\N_{\tG}(R)\widehat N$.
By the proof of \cite[Thm.~3.14]{NS14}, $(\Ind_{\widehat A,\tG}^{A_{\tchi}}(\widehat\cP),\Ind^{\N_{A}(R)_{\tvhi}}_{\widehat N,\N_{\tG}(R)}(\widehat\cP'))$ gives  
\[(A_{\tchi},\tG,\tchi)\geqslant_b(\N_A(R)_{\tvhi},\N_{\tG}(R),\tvhi).\]

Therefore, if $\chi$ runs over $\mathcal T$, then we can obtain a bijection \[\widetilde\varOmega\colon\Irr(\tG\mid\cI)\to\Alp(\tG\mid\cA).\] It can be checked directly that $\widetilde\varOmega$ is $(\Lin(\tG/G)\rtimes A)$-equivariant and preserves blocks, and satisfies that for every $\tchi\in\Irr(\tG\mid\cI)$ and $\overline{(\tR,\tvhi)}=\widetilde\varOmega(\tchi)$, \[(A_{\tchi},\tG,\tchi)\geqslant_b(\N_A(\tR)_{\tvhi},\N_{\tG}(\tR),\tvhi).\]
Thus we complete the proof.
\end{proof}

Now we generalize Proposition~\ref{prop-iso-special-ell'} and prove:

\begin{thm}\label{thm:Clifford-weights}
	Let $G$, $\tG'$ and $\tG$ be normal subgroups of a finite group $A$ such that $G\subseteq \tG'\subseteq \tG$.
	Let $B$ be a union of blocks of $G$ which is  $A$-stable and let $\tB$ be the union of blocks of $\tG$ which cover a block in $B$.
	Suppose that $\cI\subseteq \Irr(B)$, $\widetilde{\cI}\subseteq\Irr(\tG\mid\cI)\cap\Irr(\tG\mid1_{\Z(\tG')_p})$ and $\cA\subseteq \Alp(B)$ are $A$-stable such that the following conditions hold.
	\begin{enumerate}[\rm(i)]
		\item $\tG/G$ is abelian, $\C_{\tG'}(G)=\Z(\tG')$ and $\tG'/G=(\tG/G)_p$.
		\item Every  character $\chi\in\cI$ extends to  $\tG_{\chi}$.
		\item For every $\chi\in\cI$, $\Irr(\tG\mid\chi)\cap \widetilde\cI$ is a non-empty $\Lin_{p'}(\tG/G)$-orbit.
		\item There is a blockwise $A$-equivariant bijection 
		\[\varOmega\colon\cI/\!\sim_{\tG'} \ \to\ \cA/\!\sim_{\tG'}\] such that for every $\chi\in\cI$, there is a weight $(R,\vhi)$ in $\cA$ whose $\tG'$-orbit corresponds to the $\tG'$-orbit of $\chi$ via $\varOmega$ satisfying that $\N_A(R)_\vhi\subseteq A_\chi$ and 
		\[ (A_\chi,G,\chi) \geqslant_{(g),b} (\N_A(R)_\vhi,\N_G(R),\vhi)\]
 is normal with respect to $\N_{\tG'}(R)_\vhi$.		
	\end{enumerate}	
	Then there exists a blockwise $(\Lin_{p'}(\tG/G)\rtimes A)$-equivariant bijection $\widetilde\varOmega\colon\widetilde{\cI}\to\Alp(\tG\mid\cA)$ such that for every $\tchi\in\widetilde\cI$ and $\overline{(\tR,\tvhi)}=\widetilde\varOmega(\tchi)$,  \[(A_{\tchi},\tG,\tchi)\geqslant_b(\N_A(\tR)_{\tvhi},\N_{\tG}(\tR),\tvhi).\]
\end{thm}

\begin{proof}
By (iii), we have $\cI=\Irr(G\mid\widetilde\cI)$.	
Let $\chi\in\cI$ and $\overline{(R,\vhi)}\in\cA$ be as in (iv), that is, they satisfy that the $\tG'$-orbit of $\overline{(R,\vhi)}$ corresponds to $\chi$ via $\varOmega$ such that $\N_A(R)_\vhi\subseteq A_\chi$ and
\begin{equation}\label{equ-char-iso1}
	(A_\chi,G,\chi) \geqslant_{(g),b} (\N_A(R)_\vhi,\N_G(R),\vhi)
\addtocounter{thm}{1}\tag{\thethm}
\end{equation}
which is given by $(\cP,\cP')$ normally with respect to $\N_{\tG'}(R)_\vhi$.
By (b), $\vhi$ extends to $\N_{\tG}(R)_\vhi$.
Since $\varOmega$ is $A$-equivariant and $\N_A(R)_\vhi\subseteq A_\chi$, one has $A_\chi=\N_A(R)_\vhi\tG'_\chi$.
Denote by $\sigma$ the linear map defined by $(\cP,\cP')$ as in Theorem~\ref{thm:iso-char-triple}.
Also note that $\Z(G)=G\cap\Z(\tG')$.

Now $\Irr(\tG\mid\chi)\cap \widetilde\cI$ is a non-empty $\Lin_{p'}(\tG/G)$-orbit, we have that there exists a unique extension $\chi'$ of $\chi$ to $\tG'_\chi$ such that \[\Irr(\tG\mid\chi')=\Irr(\tG\mid\chi)\cap \widetilde\cI.\]
Since $\cI$ is $A$-stable, $\chi'$ is $A_\chi$-invariant. 
In addition, we have $\Z(\tG')_p\subseteq\ker(\chi')$ because $\widetilde{\cI}\subseteq\Irr(\tG\mid1_{\Z(\tG')_p})$.
Let $\hchi=\Res^{\tG'}_{G\N_{\tG'}(R)_\vhi}(\chi')$ and $\hvhi_1=\sigma^{(\chi)}_{G\N_{\tG'}(R)_\vhi}(\hchi)$.
Then $\hvhi_1$ is $\N_A(R)_\vhi$-invariant and $\Z(\tG')_p\le\ker(\hvhi_1)$.

Let $(\tR,\tvhi)$ be a weight of $\tG$ covering $(R,\vhi)$.
Then by \cite[\S2.C]{BS22}, $\tR/R\cong\tR\N_G(R)/\N_G(R)$ is a Sylow $p$-subgroup of $\N_{\tG}(R)_\vhi/\N_G(R)$, which is isomorphic to $\N_{\tG'}(R)_\vhi/\N_G(R)$.
Also \[\tR\N_G(\tR)/R\cong\N_G(\tR)/R\ti\tR/R.\]
Let $\vhi_0\in\pi_{\tR/R}(\vhi)\in\dz(\N_G(\tR)/R)$ be the DGN-correspondent of $\vhi$ and \[\hvhi'=\vhi_0\ti 1_{\tR/R}\in\Irr(\tR\N_G(\tR)\mid\vhi_0)\] so that $\tvhi$ covers $\hvhi'$.
By Lemma~3.12 and Theorem~5.13 of \cite{NS14}, there exists an $\N_A(\tR)_\vhi/R$-invariant character $\widehat{\vhi}\in\Irr(\N_G(R)\tR/R)$ such that 
\begin{equation}\label{equ-char-iso}
	(\N_A(R)_{\hvhi},\N_G(R)\tR,\hvhi)\geqslant_{b}(\N_A(\tR)_{\hvhi'},\N_G(\tR)\tR,\hvhi').	
\addtocounter{thm}{1}\tag{\thethm}
\end{equation}
In fact, $\N_A(R)_{\hvhi}=\N_A(R)_{\vhi}$, $\N_A(\tR)_{\hvhi'}=\N_A(\tR)_\vhi$ and $\hvhi$ is an extension of $\vhi$.
As $\Z(\tG')_p\le\ker(\hvhi')$, we have $\Z(\tG')_p\le\ker(\hvhi)$.

Now $\hvhi$ and $\hvhi_1$ are two $\N_A(R)_\vhi$-invariant extensions of $\vhi$.
So $\hvhi=\hvhi_1\iota$, where $\iota\in\Irr(\N_{\tG'}(R)_\vhi/\N_G(R))$ is $\N_A(R)_\vhi$-invariant.
By (a), $\C_{\tG'}(G)=\Z(\tG')$.
This gives $\hvhi(c)=\hvhi_1(c)\ne 0$ for $c\in\C_{\tG'}(G)$. Hence $\C_{\tG'}(G)\le\ker(\iota)$.
Since (\ref{equ-char-iso1}) is given by $(\cP,\cP')$ normally with respect to $\N_{\tG'}(R)_\vhi$, 
there exists a projective $\CC$-representation $\cP''$ associated with the character triple $(\N_A(R)_\vhi,\N_G(R),\vhi)$ such that
 $(\cP,\cP'')$ also gives \[(A_\chi,G,\chi) \geqslant_{(g),b} (\N_A(R)_\vhi,\N_G(R),\vhi),\] and $\cP''(n)=\iota(n)\cP'(n)$ for any $n\in \N_{\tG'}(R)_\vhi$.
We change $\cP'$ to be $\cP''$ (which is also denoted by $\cP'$) so that the corresponding linear map, which is also denoted by $\sigma$, satisfies \[\sigma_{\tG'_\chi}(\chi')=\sigma_{G\N_{\tG'}(R)}(\hchi)=\hvhi.\]
 Let $\alpha$ and $\alpha'$ be the factor sets of $\cP$ and $\cP'$ respectively.

Suppose that $\chi'$ is afforded by $\cQ\otimes\cP\rceil_{\tG'_\chi}$ where $\cQ$ is a projective $\CC$-representation of $\tG'_\chi/G$.
Then $\hvhi$ is afforded by $\cQ\rceil_{\N_{\tG'}(R)_\vhi}\otimes\cP'\rceil_{\N_{\tG'}(R)_\vhi}$.
Let $\widehat\cP$ be a projective $\CC$-representation associated with $(A_{\chi'},\tG'_\chi,\chi')$ such that $\widehat\cP\rceil_{\tG'_\chi}=\cQ\otimes\cP\rceil_{\tG'_\chi}$. 
Let $\beta$ be the factor set of $\widehat\cP$. 
As in the proof of \cite[Prop.~3.9]{NS14}, we can write \[\widehat\cP(g)=\cU(g)\otimes\mathcal P(g)\ \ \textrm{for every}\ \ g\in A_\chi,\] where $\cU$ is a projective $\CC$-representation of $A_\chi$.
Therefore, the factor of $\cU$ is $\beta\alpha^{-1}$, and $\cU\rceil_{\tG'_\chi}=\cQ$.
We define $\widehat\cP'$ to be the projective $\CC$-representation associated with $(\N_A(R)_{\hvhi},\N_{\tG'}(R)_{\hvhi},\hvhi)$ by \[\widehat\cP'(x)=\cP'(x)\otimes\cU_{\N_A(R)_\vhi}(x)\] for every $x\in\N_A(R)_{\hvhi}$.
As by a similar argument of the proof of \cite[Prop.~3.9]{NS14} (and by Lemma~\ref{lem:iso-char-triples1} and its proof), $(\widehat\cP,\widehat\cP')$ gives 
\[(A_{\chi'},\tG'_\chi,\chi')\geqslant_{b}(\N_A(R)_{\hvhi},\N_{\tG'}(R)_{\vhi},\hvhi')\]
with corresponding character triple isomorphism $\sigma'$ satisfying that $\bl(\xi)=\bl(\sigma'_J(\xi))^{J}$ for every $\tG'_\chi\le J\le A_{\chi'}$ and every $\xi\in\Irr(J\mid\hchi)$.

By the proof of Lemma~\ref{lem-partial-order}, we conclude from (\ref{equ-char-iso}) that 
\[(A_{\chi'},\tG'_\chi,\chi')\geqslant_{b}(\N_A(\tR)_{\hvhi'},\N_{\tG'}(\tR)_{\vhi},\hvhi')\]	
and similar as the proof of Proposition~\ref{prop-iso-special-ell'}, we can prove
\begin{equation}\label{equ-iso-triples-tilde}
(A_{\tchi'},\tG',\tchi')\geqslant_b(\N_A(\tR)_{\tvhi'},\N_{\tG'}(\tR),\tvhi')
\addtocounter{thm}{1}\tag{\thethm}
\end{equation}
for $\tchi'=\Ind_{\tG'_\chi}^{\tG'}(\chi')$ and $\tchi'=\Ind_{\N_{\tG'}(\tR)_\vhi}^{\N_{\tG'}(\tR)}(\vhi')$.
We remark that \[\Irr(\tG'\mid\chi)\cap \Irr(\tG'\mid\widetilde\cI)=\{\tchi'\}\] and $(\tR,\tvhi')$ is the unique weight of $\tG'$ covering $(R,\vhi)$.

Let $\cI'=\Irr(\tG'\mid\cI)\cap \Irr(\tG'\mid\widetilde\cI)$ and $\cA'=\Alp(\tG'\mid\cA)$.
Then using the above notation, \[\chi\mapsto\chi'\ \ \textrm{and} \ \ (R,\vhi)\mapsto(\tR,\tvhi')\] induce bijections $\cI/\!\sim_{\tG'}\to\cI'$ and $\cA/\!\sim_{\tG'}\to\cA'$ respectively.
From this we have a bijection $\varOmega'\colon\cI'\to\cA'$ such that  (\ref{equ-iso-triples-tilde}) holds for every $\tchi'\in\cI'$ and $\varOmega'(\tchi')=\overline{(\tR,\tvhi')}$.
In this way, we may assume that $\tG'=G$ (i.e., $\tG/G$ is a $p'$-group), and thus we can complete the proof by Proposition~\ref{prop-iso-special-ell'}.
\end{proof}


\vskip 1pc
\section*{Acknowledgements} 
The author is indebted to Jiping Zhang for useful conversations and suggestions, and to Gunter Malle for his thorough reading of the material presented here and helpful comments.

\end{document}